\documentclass[11pt,reqno]{amsart}
\usepackage{a4wide}
\usepackage{amssymb}
\usepackage{amsthm}
\usepackage{amsmath,todonotes}
\usepackage{amscd}
\usepackage{cancel}
\usepackage{verbatim}
\usepackage{color}
\usepackage[all]{xy}
\usepackage[mathscr]{eucal}
\usepackage{mathrsfs}
\usepackage{fullpage}
\usepackage{enumitem}
\usepackage{bbm}
\usepackage{qtree}
\usepackage{bm}
\usepackage{bigints}
\usepackage[headheight=110pt,top=0.7in, bottom=0.7in, left=0.6in, right=0.6in]{geometry}
\allowdisplaybreaks

\numberwithin{equation}{section}

\newcommand{\mf}[1]{\mathfrak{#1}}

\newtheorem{theorem}{Theorem}
\newtheorem{thm}[theorem]{Theorem}
\newtheorem{lemma}[theorem]{Lemma}
\newtheorem{lem}[theorem]{Lemma}
\newtheorem{corollary}[theorem]{Corollary}

\newtheorem{proposition}[theorem]{Proposition}

\theoremstyle{definition}

\numberwithin{theorem}{section} 
\numberwithin{equation}{section}
\numberwithin{table}{section}

\newcommand{\R}{\mathbb{R}}

\newcommand{\Z}{\mathbb{Z}}
\newcommand{\N}{\mathbb{N}}

\begin{document}

\title[Representation of even Gaussian integer \`a la Chen]{Representation of even Gaussian integer \`a la Chen}
\author{Soumyarup Banerjee}
\address{Department of Mathematics, Indian Institute of Technology Kharagpur, Kharagpur, West Bengal - 721302, India.}
\email{soumya.tatan@gmail.com}
\author{Habibur Rahaman}
\address{Department of Mathematics and Statistics, Indian Institute of Science Education and Resaerch Kolkata, Mohanpur, West Bengal - 741246, India.}
\email{hr21rs044@iiserkol.ac.in}



\medskip
\begin{abstract}
In this article, we represent an even Gaussian integer with sufficiently large norm as a sum of a Gaussian prime and a Gaussian integer with at most two Gaussian prime factors akin to Chen in the rational case.
\end{abstract}

\maketitle

\section{Introduction}
One of the ancient conjecture was posed by Goldbach around 1742 in a letter to Euler. The conjecture states that every even integer greater than two may be expressed as the sum of two primes. Till date, the conjecture has been verified upto the even number $4\times 10^{18}$. In literature, several progresses in different directions have been made towards this conjecture but it still remains unsolved.

In one of the directions, sieve theory has been used as a main tool to restrict the number of primes of an integer. The story begins around 1948 due to R\'enyi \cite{Renyi} who basically proved that every sufficiently large even integer $n$ can be written as sum of a prime and an integer with at most $k$ prime factors. Later, many investigations have been made towards this direction to generalize the result by reducing the value of $k$. Pan \cite{Pan},  Wang \cite{Wang} and Bombieri \cite{Bombieri} managed to improve the value of $k$ to $5$, $4$ and $3$ respectively. Finally, Chen \cite{Chen} established that every sufficiently large even number $n$ can be expressed as sum of a prime and an integer with at most $2$ prime factors, which is the best possible result towards Goldbach's conjecture. 

Several mathematicians have also studied Goldbach's conjecture in the language of a number field through sieving techniques. We call an algebraic integer $\mathcal{N} (\neq 0)$ even if each prime ideal with norm $2$ divide $\mathcal{N}$ and totally positive if all of its conjugates are positive real number. Let $\Pi_m$ denotes a totally positive algebraic integer with at most $m$ prime ideal divisors. It was Rademacher \cite{Rademacher}, who made the first impact towards Goldbach conjecture in number field by establishing every totally positive even algebraic integer $\mathcal{N}$ can be written as $\mathcal{N}=\Pi_7+\Pi_7$. Later, few improvements have been made in Rademacher's approach by expressing $\mathcal{N}=\Pi_2+\Pi_3$ due to Vinogradov \cite{Vinogradov} and $\mathcal{N}=\Pi_1+\Pi_{m_0}$ for some fixed $m_0$ due to Hinz \cite{Hinz1} respectively. Finally around 1991, Hinz \cite{Hinz2} investigated this problem in totally real algebraic number field and obtained an analogue of Chen's theorem for totally positive even algebraic integer $\mathcal{N}$ by expressing $\mathcal{N}=\Pi_1+\Pi_2$.
 
Holben and Jordan \cite{Holben} around 1968 conjectured Goldbach's problem for the ring of Gaussian integers, which precisely states as for every even Gaussian integer $\mathcal{N}$, there are Gaussian primes $p_1$ and $p_2$ such that $\mathcal{N} = p_1+p_2$. In this manuscript, we have studied the above conjecture for every even Gaussian integer with sufficiently large norm and the primary concern here is to represent an even Gaussian integer in an analogues way of Chen's representation in rational case. 

Let $r(\mathcal{N})$ be the number of representation of $\mathcal{N}$ as a sum of a Gaussian prime and a Gaussian integer with at most two Gaussian prime ideal factors. For any Gaussian integer $n$, we denote its usual norm by $N(n)$. Our main goal here is to find the lower bound of $r(\mathcal{N})$ for an even Gaussian integer $\mathcal{N}$ with sufficiently large $N(\mathcal{N})$.
\begin{thm}\label{main_thm}
Let $\mathcal{N}$ be an even Gaussian integer. Then we have the lower bound 
				\begin{align*}
					r(\mathcal{N})\gg \mathfrak{S}_1(\mathcal{N})\frac{N(\mathcal{N})}{(\log N(\mathcal{N}))^2},
				\end{align*}
				where 
				\begin{align}\label{Sing series}
					\mathfrak{S}_1(\mathcal{N}):=\displaystyle\prod_{N(p)>2}\left(1-\frac{1}{(N(p)-1)^2}\right)\prod_{\substack{N(p) >2\\ p\mid \mathcal{N}}}\frac{N(p)-1}{N(p)-2}.
				\end{align}
\end{thm}
We continue here by making remark about the key tools we have used to prove the above theorem. Few of the important arguments are influenced by the method of proving Chen's corresponding result in rational case by Nathanson \cite{Nathanson}. 

Selberg's weighted sieve method allows us to obtain a lower bound of $r(\mathcal{N})$ in terms of three sieving functions and we have mainly applied linear sieve inequality for Gaussian integers to bound those sieving functions. Two essential ingredients to estimate the error terms are an analogue of Bombieri-Vinogradov type theorem for Gaussian integers and an analogue of large sieve inequality in number fields due to Huxley. 

The following corollary is a direct application of our main result.
\begin{corollary}
Any even Gaussian integer with sufficiently large norm can be expressed as a sum of a Gaussian prime and a Gaussian integer with at most two Gaussian prime factors.
\end{corollary}
The paper is organized as follows. In Section \ref{Background set up} we fix few notations and collect essential results for Gaussian integers which are important throughout the article. In Section \ref{Lower bound of r(N)} we obtain the lower bound of $r(\mathcal{N})$ in terms of three sieving functions. In Section \ref{Linear Sieve in Z[i]} we set basic ingredients to apply linear sieve on the sieving functions appeared in Section \ref{Lower bound of r(N)}. We obtain the desired bound of the sieving functions in Section \ref{Estimation of S(A, P, z)}, \ref{Estimation of S(A_q, P, z)} and \ref{Estimation of S(B, P, y)} respectively. Finally, in Section \ref{Proof of Main Theorem}, we combine the bounds of the previous sections in order to prove Theorem \ref{main_thm}.

\section{Background set up}\label{Background set up}
In this section, we fix some notations and state important results for Gaussian integers which are essential in our exposition. We denote the set of Gaussian integers by its usual notation $\Z[i]$. Let $a, b, c \in \Z[i]$ and the principal ideals generated by $a, b$ and $c$ be denoted by $\mf{a}, \mf{b}$ and $\mf{c}$ respectively.  We call $a$ is divisible by $c$ if $\mf{a}\subseteq \mf{c}$ and $a\equiv b \pmod c$ if $a-b\in \mf{c}$. Recall that $\Z[i]$ is a principal ideal domain, which implies every ideal is generated by some Gaussian integer. Thus the ideal norm of the ideal $\mf{a}$, which we denote by $N(\mf{a})$ is same as the usual norm of the Gaussian integer $a$ and for that we will frequently interchange both of the norms as per requirements. For non-zero $a, c \in \Z[i]$, a greatest common divisor of $a$ and $c$ is a common divisor with maximal norm. Similar definition also holds for a greatest common divisor of two ideals $\mf{a}$ and $\mf{c}$. We abbreviate the notation for the norm of both of the greatest common divisors as $N(\gcd(a, c)) = N(a, c)$ and $N(\gcd(\mf{a}, \mf{c})) = N(\mf{a}, \mf{c})$.

Throughout, we use $p, q$ to denote Gaussian primes and $\mf{p}, \mf{q}, \mf{n}, \mf{d}$ to denote the principal ideals generated by $p, q, n, d$ respectively in $\Z[i]$. We next recall two multiplicative functions of an ideal in $\Z[i]$. The M\"{o}bius function $\mu$ on an ideal $\mathfrak{n}\subset\mathbb{Z}[i]$ can be defined as
		\begin{align}\label{Mobius}
			\mu(\mathfrak{n})=\begin{cases}
				(-1)^r,  \text{   if }\mathfrak{n}=\mathfrak{p}_1\mathfrak{p}_2\ldots\mathfrak{p}_r\\
				0,  \text{   if $\mf{n}$ is not square free,}
			\end{cases}
		\end{align}
where $\mathfrak{p}_1,\mathfrak{p}_2\ldots,\mathfrak{p}_r$ are prime ideals and the Euler totient function $\phi$ on an ideal $\mathfrak{n}\subset\mathbb{Z}[i]$ can be defined as
			\begin{align}\label{phin}
\phi(\mathfrak{n})=\frac{1}{N(\mathfrak{n})}\prod_{\mathfrak{p}\mid\mathfrak{n}}\left(1-\frac{1}{N(\mathfrak{p})}\right),
			\end{align}
where the product runs over the prime ideals. The following lemma provides an analogue of Merten's result for rational case in $\Z[i]$ set up.
\begin{lemma}\label{Merten's theorem_in_z[i]}
	For sufficiently large $x$, we have
	\begin{align*}
	\sum_{\substack{\N(\mf{p})<x}}\frac{1}{N(\mf{p})}=\log\log x+B+O\left(\frac{1}{\log x}\right),
	\end{align*}
	for some constant $B>0$ and 
	\begin{align}\label{Merten_2}
		\prod_{\substack{N(\mf{p})<x}}\left(1-\frac{1}{N(\mf{p})}\right)^{-1}=\frac{\pi}{4}e^\gamma\log x\left(1+O\left(\frac{1}{\log x}\right)\right).
	\end{align}
\end{lemma}
We refer \cite[Theorem~7.153]{Bordelles} for the details of the above lemma w. One can observe that the factor $\frac{\pi}{4}$ occurs in \eqref{Merten_2} by comparing \cite[Equation~(7.21)]{Bordelles} and \cite[Page~454]{Akshaa}. The next lemma follows from \eqref{Merten_2} which is also an analogue of one of the Mertern's result for rational case in $\Z[i]$.

\begin{lemma}\label{Lem:mertern}
				For any $\epsilon>0$, there is a real number $u_0(\epsilon)>0$ such that 
				\begin{align}\label{Pre_sieve_inequality}
					\prod_{u\leq N(\mathfrak{p})<z}\left(1-\frac{1}{N(\mathfrak{p})}\right)^{-1}<(1+\epsilon)\frac{\log z}{\log u}
				\end{align}
holds for any $u_0(\epsilon)\leq u<z.$
			\end{lemma}
\begin{proof}
It follows from \eqref{Merten_2} that for $x>2$,
\begin{align*}
\prod_{\substack{\mathfrak{p}\\N(\mathfrak{p})< x}}\left( 1-\frac{1}{N(\mathfrak{p})}\right)^{-1}\sim C\log x,
\end{align*}
where $C:=\frac{\pi}{4}e^\gamma>0$. We choose $\delta(\epsilon)>0$ such that 
	\begin{align}\label{C_delta}
		\frac{C+\delta}{C-\delta}<(1+\epsilon).
	\end{align}
Therefore, there exists some $u_0(\epsilon)>0$ such that for all $x\geq u_0(\epsilon)$,
\begin{align*}
	(C-\delta)\log x<\prod_{\substack{\mathfrak{p}\\N(\mathfrak{p})< x}}\left( 1-\frac{1}{N(\mathfrak{p})}\right)^{-1}<(C+\delta)\log x
\end{align*}	
Thus, we can write for any $u_0(\epsilon)\leq u<z$,
\begin{align*}
		\prod_{\substack{\mathfrak{p}\\u\leq N(\mathfrak{p})< z}}\left( 1-\frac{1}{N(\mathfrak{p})}\right)^{-1}&=\frac{	\prod\limits_{\substack{\mathfrak{p}\\N(\mathfrak{p})< z}}\left( 1-\frac{1}{N(\mathfrak{p})}\right)^{-1}}{	\prod\limits_{\substack{\mathfrak{p}\\N(\mathfrak{p})< u}}\left( 1-\frac{1}{N(\mathfrak{p})}\right)^{-1}}\\
		&=\frac{(C+\delta)\log z}{(C-\delta)\log u}\\
		&<(1+\epsilon)\frac{\log z}{\log u},
\end{align*}
where the last step follows from \eqref{C_delta}.
	\end{proof}
For any $x>0$, let $\pi (x)$ denotes the number of ideals  in $\mathbb Z[i]$ with norm $\leq x$. In the following lemma we state the result due to Landau which is known as the prime ideal theorem (cf. \cite[Theorem~7.151]{Bordelles})
\begin{lem}\label{PIT}
	We have for sufficiently large $x$
	\begin{align*}
		\pi(x)=\frac{x}{\log x}+O\left(\frac{x}{(\log x)^2}\right).
	\end{align*}
\end{lem} 
Let $\pi (x;d,a)$ denote the number of primes $p$ in $\mathbb{Z}[i]$ with $N(p)\leq x$ such that $p\equiv a(\bmod\ d)$, where $a, d\in\mathbb{Z}[i]$ with $N(a,d)=1$. We next state an analogue of Bombieri-Vinogradov theorem in $\mathbb{Z}[i]$.
\begin{lemma}\label{B-V thm}
For any $A>0$, there exists $B_A>0$ such that 
\begin{align*}
\sum_{4<N(d)\leq \frac{x^{1/2}}{(\log x)^{B_A}}}\left|\pi(x;d,a)-\frac{4\pi(x)}{\phi(\mathfrak{d})}\right|\ll \frac{x}{(\log x)^A}.
\end{align*}
\end{lemma}
For details one can see \cite{Huxley} and \cite[Lemma 10.2]{Akshaa} as a reference.
\section{Lower bound of $r(\mathcal{N})$}\label{Lower bound of r(N)}
In this section, we mainly bound $r(\mathcal{N})$ in terms of three sieving functions and for that we set $\mathcal{P}$ to be the set of all Gaussian primes that do not divide $\mathcal{N}$ and fix 
\begin{equation*}
\mathcal{A} := \{\ \mathcal{N}-p: p\in\mathcal{P}, N(p)< N(\mathcal{N})\}.
\end{equation*}
Clearly, 
\begin{equation}\label{A card}
|\mathcal{A}|=4\pi(N(\mathcal{N}))-4\omega(\mathcal{N}),
\end{equation}
where $\omega(\mathcal{N})$ denotes the number of distinct prime divisors of $\mathcal{N}$ up to unit. For any real number $z\geq 2 $, we denote 
$$P(z):= \prod_{\substack{\mathfrak{p}\\p\in\mathcal{P}\\N(p)<z}}\mathfrak{p}.$$
Let $S(\mathcal{A}, \mathcal{P}, z)$  denotes the cardinality of the set $\left\{n\in\mathcal{A}:N(\mathfrak{n},P(z))=1\right\}$. 
We define
\begin{align*}
w(n) := 1-\frac{1}{2}\sum_{\substack{\mathfrak{q}\\z\leq N(q)<y\\ \mathfrak{q}^k||\mathfrak{n}}}k-\frac{1}{2}\sum_{\substack{\mathfrak{p}_1,\mathfrak{p}_2,\mathfrak{p}_3\\\mathfrak{p}_1 \mathfrak{p}_2 \mathfrak{p}_3=\mathfrak{n}\\z\leq N(p_1)<y\leq N(p_2)\leq N(p_3)}}1
\end{align*}
for every $n\in\mathbb{Z}[i]$. The following lemma provides a lower bound of $r(\mathcal{N})$.
\begin{lemma}\label{Lem:r(N) inequality}
For any real $z\geq 2 $ and $y = N(\mathcal{N})^{1/3}$, the inequality
\begin{align}\label{r(N) inequality}
r(\mathcal{N})\geq S(\mathcal{A}, \mathcal{P}, z)-\frac{1}{2}T_1-\frac{1}{2}T_2 -4
\end{align}
holds, where 
\begin{align*}
T_1 := \sum_{\substack{n\in\mathcal{A}\\N(\mathfrak{n},P(z))=1}}\sum_{\substack{z\leq N(\mf{q})<y\\ \mathfrak{q}^k||\mathfrak{n}}}k  \ \ \text{and } \ \
T_2:= \sum_{\substack{n\in\mathcal{A}\\N(\mathfrak{n},P(z))=1}}\sum_{\substack{\mathfrak{p}_1,\mathfrak{p}_2,\mathfrak{p}_3\\\mathfrak{p}_1\mathfrak{p}_2\mathfrak{p}_3=\mathfrak{n}\\z\leq N(\mf{p}_1)<y\leq N(\mf{p}_2)\leq N(\mf{p}_3)}}1.
\end{align*} 
\end{lemma}
\begin{proof} 
We first consider a Gaussian integer $n$ satisfying $N(n)<N(\mathcal{N})$ with $N(n,\mathcal{N}) =N(\mathfrak{n},P(z))=1$. Then clearly, the ideal $\mathfrak{n}$ is divisible only by prime ideals $\mathfrak{p}$ with $N(\mathfrak{p})\geq z$. Thus, $\mathfrak{n}$ can be decomposed in the form $\mathfrak{n}=\mathfrak{p}_1\mathfrak{p}_2 \ldots\mathfrak{p}_r\mathfrak{p}_{r+1}\ldots \mathfrak{p}_{r+s}$ with $z\leq N(\mathfrak{p}_1)\leq\ldots\leq N(\mathfrak{p}_r)<y\leq N(\mathfrak{p}_{r+1})\leq\ldots N(\mathfrak{p}_{r+s})$ for some $r, s>0$. For $y = N(\mathcal{N})^{1/3}$, we can write $N(\mathcal{N})^{s/3}=y^s\leq N(\mathfrak{p}_{r+1})N(\mathfrak{p}_{r+2})\ldots N(\mathfrak{p}_{r+s})\leq N(\mathfrak{n})<N(\mathcal{N})$ which implies $s\in\{0,1,2\} $. We also have from the prime ideal decomposition of $\mathfrak{n}$ that
$$\frac{1}{2}\sum_{\substack{z\leq N(\mathfrak{q})<y\\ \mathfrak{q}^k||\mathfrak{n}}}k=\frac{r}{2},$$ which implies $r/2<1$ for $w(n)>0$. Thus $r$ can take only the values $0$ and $1$ for $w(n)>0$. Now if $r=1$ and $s=2$  then $\mathfrak{n}=\mathfrak{p}_1\mathfrak{p}_2\mathfrak{p}_3$ with $z\leq N(\mathfrak{p}_1)<y\leq N(\mathfrak{p}_2)\leq N(\mathfrak{p}_3)$, which implies that $w(n)=0$. Hence, for $w(n)>0$, either $r=0 \text{ and } s\in\{0,1,2\}$, or $r=1 \text{ and } s\in\{0,1\}$. Therefore, the ideal $\mathfrak{n}$ generated by a Gaussian integer $n$ satisfying $N(n)<N(\mathcal{N})$, $N(n,\mathcal{N})=N(\mathfrak{n},P(z))=1$ and $w(n)>0$ is either having norm $1$, or of the form $\mathfrak{p}_1, \mathfrak{p}_1 \mathfrak{p}_2$ such that $N(\mathfrak{p}_2)\geq N(\mathfrak{p}_1)\geq z$ and $p_1, p_2$ does not divide $\mathcal{N}$.

For all $n=\mathcal{N}-p\in\mathcal{A}$, we have $N(n,\mathcal{N})=1$ since if $N(n,\mathcal{N})>1$, then there exist a prime $q$ such that $q\mid n$ and $q\mid \mathcal{N} $, which together implies $q\mid p$, but then $p\mid \mathcal{N}$, which is not possible. 

Letting $\mathbb{H}:=\{n \in \Z[i]: N(\mathfrak{n})=1 \text{ or } \mathfrak{n} \text{ is one of }\mathfrak{p}_1,\mathfrak{p}_1\mathfrak{p}_2 \text{ with } N(\mathfrak{p}_2)\geq N(\mathfrak{p}_1)\geq z\}$, it follows from the above arguments that we can bound $r(\mathcal{N})$ as
			\begin{align*}
				r(\mathcal{N})+4 \geq \displaystyle\sum_{\substack{\mathcal{N}=p+n\\\mathfrak{n}\in \mathbb{H}}}1
				 \geq \displaystyle\sum_{\substack{n\in\mathcal{A}\\\mathfrak{n}\in\mathbb{H}}}1
			\end{align*}
The definition of $w(n)$ directly implies $w(n)\leq 1$, which yields		
			\begin{align*}				 
				r(\mathcal{N})+4 & \geq\displaystyle\sum_{\substack{n\in\mathcal{A}\\N(\mathfrak{n},P(z))=1}}w(n)\\ 
				&=\left(\sum_{\substack{n\in\mathcal{A}\\N(\mathfrak{n},P(z))=1}}1\right)-\frac{1}{2}\left(\sum_{\substack{n\in\mathcal{A}\\N(\mathfrak{n},P(z))=1}}\sum_{\substack{z\leq N(\mf{q})<y\\ \mathfrak{q}^k||\mathfrak{n}}}k \right)-\frac{1}{2}\left(\sum_{\substack{n\in\mathcal{A}\\N(\mathfrak{n},P(z))=1}}\sum_{\substack{\mathfrak{p}_1,\mathfrak{p}_2,\mathfrak{p}_3\\\mathfrak{p}_1\mathfrak{p}_2\mathfrak{p}_3=\mathfrak{n}\\z\leq N(\mf{p}_1)<y\leq N(\mf{p}_2)\leq N(\mf{p}_3)}}1 \right).
			\end{align*}
Here the first sum on the right hand side is exactly same as $S(\mathcal{A}, \mathcal{P}, z)$, which concludes the proof of the lemma.
\end{proof}
We next bound $T_1$ and $T_2$ to obtain the lower bound of $r(\mathcal{N})$. The next lemma provides the upper bound of $T_1$.
\begin{lemma}\label{T_1 sum}
Let $y, z$ be any real with $y\geq z \geq 2$. Then for $\mathcal{A}_\mathfrak{q}=\{n\in\mathcal{A}:\mathfrak{q}\mid \mathfrak{n}\}$, we have the bound
\begin{align}\label{T2}
T_1 \leq \sum_{\substack{z\leq N(\mf{q})<y}}S(\mathcal{A}_{\mf{q}}, \mathcal{P}, z) + O\left(\frac{N(\mathcal{N})}{z}\right).
\end{align}
\end{lemma}
\begin{proof}
We can split the sum $T_1$ in \eqref{r(N) inequality} into two parts, namely			
\begin{align}\label{T2 sum}
				T_1
				&=\sum_{\substack{n\in\mathcal{A}\\N(\mathfrak{n},P(z))=1}}\sum_{\substack{z\leq N(\mf{q})<y\\ \mathfrak{q}\mid \mathfrak{n}}}1 +\sum_{\substack{n\in\mathcal{A}\\N(\mathfrak{n},P(z))=1}}\sum_{\substack{z\leq N(\mf{q})<y\\ \mathfrak{q}^k||\mathfrak{n}\\k\geq 2}}(k-1)\nonumber\\
				&=\sum_{\substack{z\leq N(\mf{q})<y}}\sum_{\substack{n\in\mathcal{A}\\\mathfrak{q}|\mathfrak{n}\\N(\mathfrak{n},P(z))=1}}1+\sum_{\substack{n\in\mathcal{A}\\N(\mathfrak{n},P(z))=1}}\sum_{\substack{z\leq N(\mf{q})<y\\ \mathfrak{q}^k||\mathfrak{n}\\k\geq 2}}(k-1)\nonumber\\
				&=\sum_{\substack{z\leq N(\mf{q})<y}}S(\mathcal{A}_{\mf{q}}, \mathcal{P}, z)+T'_{1},
			\end{align}
			where $T'_{1}=\displaystyle\sum_{\substack{n\in\mathcal{A}\\N(\mathfrak{n},P(z))=1}}\sum_{\substack{z\leq N(\mf{q})<y\\ \mathfrak{q}^k||\mathfrak{n}\\k\geq 2}}(k-1)$.  
Interchanging the order of summation, the upper bound of $T'_1$ can be evaluated as 
			\begin{align*}
				T'_{1}&=\sum_{\substack{z\leq N(\mf{q})<y}}\sum_{k=2}^{\infty}\sum_{\substack{n\in\mathcal{A}\\N(\mathfrak{n},P(z))=1\\\mathfrak{q}^k\mid \mid \mathfrak{n}}}(k-1)\\
				&\leq\sum_{\substack{z\leq N(\mf{q})<y}}\sum_{k=2}^{\infty}\sum_{\substack{N(n)<4N(\mathcal{N}) \\\mathfrak{q}^k\mid \mid \mathfrak{n}}}(k-1)\\
				&\leq\sum_{\substack{z\leq N(\mf{q})<y}}\sum_{k=2}^{\infty}(k-1) \sum_{\substack{N(n)<\frac{4N(\mathcal{N})}{N(q)^k}}}1
			\end{align*}  
The result of Weber [cf. \cite[p.144]{Murty}, \cite[p.454]{Akshaa}] leads to the asymptotic formula for the number of integral ideals with norm $\leq x$, which can be stated as 
\begin{align}\label{Weber}
\#\{n\in \Z[i] : N(n)\leq x\} = \pi x+ O(\sqrt{x}).
\end{align}		
Thus, applying \eqref{Weber} in the last sum of the above triple sum, the bound of $T'_2$ reduces to			
			\begin{align*}
				T'_{1}& \leq \sum_{\substack{z\leq N(\mf{q})<y}}\sum_{k=2}^{\infty}(k-1)\left(4\pi \frac{N(\mathcal{N})}{N(q)^k}+O
				\left(\sqrt\frac{N(\mathcal{N})}{N(q)^k}\right)\right),\\
				& =4\pi N(\mathcal{N})\sum_{\substack{z\leq N(\mf{q})<y}}\sum_{k=2}^{\infty}\frac{(k-1)}{N(q)^k}+O\left(\sqrt{N(\mathcal{N})}\sum_{\substack{z\leq N(\mf{q})<y}}\sum_{k=2}^{\infty}\frac{(k-1)}{N(q)^{k/2}}\right)\\
				&=4\pi N(\mathcal{N})\sum_{\substack{z\leq N(\mf{q})<y}}\frac{1}{(N(q)-1)^2}+O\left(\sqrt{N(\mathcal{N})}\sum_{\substack{z\leq N(\mf{q})<y}}\frac{1}{(\sqrt{N(q)}-1)^2}\right)\\
				&\ll \frac{N(\mathcal{N})}{z-2}+O\left(\frac{\sqrt{N(\mathcal{N})}}{\sqrt{z}-2}\right)\\
				& \ll \frac{N(\mathcal{N})}{z}.
			\end{align*}
Inserting the above bound of $T'_1$ into \eqref{T2 sum}, we obtain the bound $T_2$. This completes the proof of lemma.
\end{proof}
Our next task is to estimate the sum $T_3$ and for that we define
\begin{align}\label{Def of B}
\mathcal{B}:=\{\mathcal{N}-p_1 p_2p_3 :z\leq N(p_1)<y\leq N(p_2)\leq N(p_3), N(p_1 p_2p_3)<4N(\mathcal{N}),N(p_1 p_2p_3, \mathcal{N})=1\}.
\end{align}
In the next lemma, we provide the bound of $T_2$.
\begin{lemma}
For any real $y, z$ with $y\geq z \geq 2$, we have
\begin{align}\label{T3}
T_2 \leq \frac{1}{4^3}S(\mathcal{B},\mathcal{P},y)+ O(y).
\end{align}
\end{lemma}
\begin{proof}
 We write the sum $T_2$ as
\begin{align*}
T_2=\sum_{\substack{n\in\mathcal{A}\\N(\mf{n},P(z))=1}}\sum_{\substack{\mathfrak{p}_1,\mathfrak{p}_2,\mathfrak{p}_3\\\mathfrak{p}_1\mathfrak{p}_2\mathfrak{p}_3=\mathfrak{n}\\z\leq N(p_1)<y\leq N(p_2)\leq N(p_3)}}1
				=\frac{1}{4^3}\sum_{\substack{p_1,p_2,p_3\\p_1 p_2 p_3\in\mathcal{A}\\z\leq N(p_1)<y\leq N(p_2)\leq N(p_3)}}1
\end{align*}
It follows from the definition of $\mathcal{A}$ that for  $n=p_1 p_2p_3\in\mathcal{A}$ with $z\leq N(p_1)<y\leq N(p_2)\leq N(p_3)$, there exist a prime $p\in\mathcal{P}$ such that $n=\mathcal{N}-p$ with $N(p)<N(\mathcal{N})$ and $p\nmid \mathcal{N}$. Clearly, $$N(n)=N(p_1 p_2p_3)=N(\mathcal{N}-p)<4N(\mathcal{N})$$
and $N(n,\mathcal{N})=N(p_1 p_2p_3, \mathcal{N})=1$. Thus it follows from the definition \eqref{Def of B} of  $\mathcal{B}$ that $p=\mathcal{N}-p_1 p_2p_3$ is an element of $\mathcal{B}$. Applying the above fact, $T_2$ can be bounded as
			\begin{align*}
				T_2 =\frac{1}{4^3}\sum_{p\in\mathcal{B}}1
				    &=\frac{1}{4^3}\sum_{\substack{p\in\mathcal{B}\\N(p)<y}}1+\frac{1}{4^3}\sum_{\substack{p\in\mathcal{B}\\N(p)\geq y}}1\nonumber\\
				    &\leq \frac{1}{4^3} \sum_{N(n)<y}1+\frac{1}{4^3}\sum_{\substack{n\in\mathcal{B}\\N(\mf{n},P(y))=1}}1\nonumber\\
				    &= \frac{1}{4^3}S(\mathcal{B},\mathcal{P},y)+ O(y),
			\end{align*}
where in the penultimate step we have applied \eqref{Weber} in the first sum and the definition \eqref{Def of B} of  $\mathcal{B}$ in the second sum. This completes the proof of the lemma.
\end{proof}
The goal of this section is to prove the following proposition.
\begin{proposition}\label{lowerbound of r(N)}
For any real $z\geq 2 $ and $y = N(\mathcal{N})^{1/3}$, we have the lower bound
\begin{align}\label{lowerbound of r(N)1}
r(\mathcal{N})\geq S(\mathcal{A},\mathcal{P},z)-\frac{1}{2}\sum\limits_{\substack{z\leq N(\mf{q})<y}}S(\mathcal{A}_\mathfrak{q}, \mathcal{P},z)-\frac{1}{128}S(\mathcal{B},\mathcal{P},y)+O\left(\frac{N(\mathcal{N})}{z}\right)+ O(y).
\end{align}
\end{proposition}
\begin{proof}
We invoke the bound \eqref{T2} and \eqref{T3} of $T_1$ and $T_2$ into Lemma \ref{Lem:r(N) inequality} to conclude the proposition.
\end{proof}
\section{Linear Sieve in $\Z[i]$}\label{Linear Sieve in Z[i]}
This section concerns in obtaining the linear sieve inequality in general settings over $\Z[i]$ and fixing basic notations to deal with the sieving function appeared in Proposition \ref{lowerbound of r(N)}. Let $\mathcal{A}\subset \mathbb{Z}[i]$ be a finite set and $\mathcal{P}$ be a set of primes in $\mathbb{Z}[i]$. For $z\geq 2 $, let 
$$P(z)= \prod_{\substack{\mathfrak{p}\\p\in\mathcal{P}\\N(p)<z}}\mathfrak{p}.$$
Let $\mathcal{A}_\mathfrak{d}=\{n\in\mathcal{A}: \mathfrak{d}|\mathfrak{n}\}$ and $g(\mathfrak{d})$ be a real valued multiplicative function defined on ideals of $\mathbb{Z}[i]$ with $0\leq g(\mathfrak{p})<1$ for all $p\in\mathcal{P}$ such that $\sum_{n\in\mathcal{A}}g(\mathfrak{d})$ approximates to $|\mathcal{A}_\mathfrak{d}|$. Thus, one can write
\begin{equation}\label{A_d}
|\mathcal{A}_\mathfrak{d}| = \sum_{n\in\mathcal{A}}g(\mathfrak{d}) + r(\mathfrak{d}),
\end{equation} 
where $r(\mathfrak{d})$ denotes the reminder term of $|\mathcal{A}_\mathfrak{d}|$. Define
\begin{align}\label{Def of V(z)}
V(z):=\prod_{\mathfrak{p}\mid P(z)}(1-g(\mathfrak{p})).
\end{align}  
We next state an analogue of Jurkat-Richert theorem on $\Z[i]$. To our best knowledge, it is not known in literature.  

\begin{lemma}\label{Jurkat-Richert}
 Let $\mathcal{Q}$ be any finite subset of $\mathcal{P}$ and $Q$ be the norm of the products of primes in $\mathcal{Q}$. For some $\epsilon\in(0,\frac{1}{200})$, let the arithmetic function $g$ satisfies the inequality
			\begin{align}\label{linear sieve inequality}
				\prod_{\substack{\mathfrak{p}\\p\in(\mathcal{P}\setminus\mathcal{Q})}}(1-g(\mathfrak{p}))^{-1}<(1+\epsilon)\frac{\log z}{\log u},
			\end{align} 
for all $1<u<z$. Then for $s=\frac{\log D}{\log z}$, we have
			\begin{align}\label{Upper bound of JR thm}
				S(\mathcal{A}, \mathcal{P}, z)<(F(s)+\epsilon e^{14-s})V(z)|\mathcal{A}|+\sum_{\substack{\mathfrak{d}|P(z)\\N(\mathfrak{d})<DQ}}|r(\mathfrak{d})|,
			\end{align}
for any $D\geq z$ and 
			\begin{align}\label{Lower bound of JR thm}
				S(\mathcal{A}, \mathcal{P}, z)>(f(s)-\epsilon e^{14-s})V(z)|\mathcal{A}|-\sum_{\substack{\mathfrak{d}|P(z)\\N(\mathfrak{d})<DQ}}|r(\mathfrak{d})|
			\end{align} 
for any $D\geq z^2$. Here $F(s), f(s)$ are the functions satisfying 
			\begin{align}\label{F(s) and f(s)}
				&sF(s)=2e^{\gamma} \ \ \text{for} \ \ 1\leq s\le3,\	\	\ (sF(s))'= f(s-1),\ \ \text{for} \ \ s>3,\nonumber\\
		\text{and}	\ \	&sf(s)=2e^{\gamma}\log(s-1) \ \ \text{for} \ \ 2\leq s\le4, \	\	\	\ (sf(s))'=F(s-1) \ \ \text{for}\ \ s>2.
			\end{align}
		\end{lemma}

\begin{proof}
The proof follows along the similar direction of Jurkat-Richert theorem on $\Z$. For the sake of completeness we provide an outline of the proof. The M\"obius function $\mu$ defined in \eqref{Mobius}, satisfies the property
\begin{align*}			
\sum_{\mathfrak{d}|\mathfrak{n}}\mu(\mathfrak{d})= 
\begin{cases}
1, \text{  if } N(\mathfrak{n})=1\\
0, \text{ otherwise }, 
\end{cases}
\end{align*}
which can be applied to write $S(\mathcal{A}, \mathcal{P}, z)$ as
\begin{align*}
S(\mathcal{A}, \mathcal{P}, z)=\sum_{\substack{n\in\mathcal{A}\\ N(\mathfrak{n}, P(z))=1}}1
=\sum_{n\in\mathcal{A}}\sum_{\mathfrak{d}|(\mathfrak{n},P(z))}\mu(\mathfrak{d})
=\sum_{\mathfrak{d}|P(z)}\mu(\mathfrak{d})\sum_{\substack{n\in\mathcal{A}\\\mathfrak{d}|\mathfrak{n}}}1=\sum_{\mathfrak{d}|P(z)}\mu(\mathfrak{d})|\mathcal{A}_\mathfrak{d}|.
\end{align*}
Inserting \eqref{A_d} in the above equation, one can reduce the above equation as
 \begin{align*}
S(\mathcal{A}, \mathcal{P}, z)&= \sum_{n\in\mathcal{A}}V(z)+R(z),
\end{align*}
where $R(z)=\displaystyle\sum_{\mathfrak{d}|P(z)}\mu(\mathfrak{d})r(\mathfrak{d})$. The key tool to bound $S(\mathcal{A}, \mathcal{P}, z)$ is Rosser's weight on ideals in $\Z[i]$, which we define next. 

For $D>0$ and $\mathfrak{d}$ be square-free ideal of the form $\mathfrak{d}=\mathfrak{p}_1\mathfrak{p}_2\ldots\mathfrak{p}_r$ with $N(\mathfrak{p}_1)> N(\mathfrak{p}_2)> \cdots>N(\mathfrak{p}_r)$, we define
\begin{align*}
&\lambda^{+}(\mathfrak{d}):=\begin{cases}
				(-1)^r &\text{  if  } N(\mathfrak{p}_1)\cdots N(\mathfrak{p}_{2\ell}) N(\mathfrak{p}_{2\ell+1})^3<D \text{ whenever } \ 0 \leq \ell \leq \frac{r-1}{2}\\
				\ \ 0  &\text{  otherwise,}
			\end{cases}\nonumber\\
&\lambda^{-}(\mathfrak{d}):=\begin{cases}
				(-1)^r &\text{  if  }N(\mathfrak{p}_1)\cdots N(\mathfrak{p}_{2\ell-1}) N(\mathfrak{p}_{2\ell})^3<D \text{ whenever } \  0 \leq \ell \leq \frac{r}{2}\\
				\ \ 0  &\text{  otherwise}
			\end{cases}			
\end{align*} 
and $\lambda^{\pm}(\mathfrak{d})=0$ if $\mathfrak{d}$ is not square-free. We can now reduce the size of error term by replacing $\mu$ with the above weights. Thus $S(\mathcal{A}, \mathcal{P}, z)$ can be bounded as
\begin{align}\label{bound Sapz}
\sum_{n\in\mathcal{A}}G(z,\lambda^{-}) +\sum_{\substack{\mathfrak{d}|P(z)\\N(\mathfrak{d})<D}}\lambda^{-}(\mathfrak{d})r(\mathfrak{d})\leq S(\mathcal{A}, \mathcal{P}, z)\leq \sum_{n\in\mathcal{A}}G(z,\lambda^{+}) +\sum_{\substack{\mathfrak{d}|P(z)\\N(\mathfrak{d})<D}}\lambda^{+}(\mathfrak{d})r(\mathfrak{d}),
\end{align}
where $G(z,\lambda^{\pm}):= \sum\limits_{\mathfrak{d}|P(z)}\lambda^{\pm}(\mathfrak{d})g(\mathfrak{d})$. Our next goal is to bound $G(z,\lambda^{\pm})$. $V(z)$ satisfies the recurrence relation   
\begin{align*}
V(z) = 1-\sum_{\mathfrak{p}\mid P(z)} g(\mathfrak{p})V(N(\mathfrak{p})).
\end{align*} 
The above relation with well-known inclusion-exclusion principle yields
\begin{align}
&G(z, \lambda^{+})=V(z)+\displaystyle\sum_{\substack{k=1\\ k\equiv1(\bmod 2)}}^{\infty}T_{k}(D,z)\label{Gz+}\\
&G(z, \lambda^{-})=V(z)-\displaystyle\sum_{\substack{k=1\\ k\equiv0(\bmod 2)}}^{\infty}T_{k}(D,z),\label{Gz-}
\end{align}
where
\begin{align*}
T_{k}(D,z):=\displaystyle\sum_{\substack{\mathfrak{p}_1,\mathfrak{p}_2,\ldots,\mathfrak{p}_k\\ y_k\leq N(\mathfrak{p}_k)<\ldots<N(\mathfrak{p}_1)<z\\ N(\mathfrak{p}_m)<y_m\ \forall m<k,\  m\equiv k(\bmod 2)}}g(\mathfrak{p}_1\mathfrak{p}_2\ldots \mathfrak{p}_k)V(\mathfrak{p}_k)
\end{align*}
and $y_m$'s are suitable parameters defined as $y_m:= \left(\frac{y}{N(\mathfrak{p}_1)N(\mathfrak{p}_2)\cdots N(\mathfrak{p}_m)}\right)^{1/2}$. We next find the upper bound of $T_{k}(D,z)$ and for that we mainly follow the similar argument as in \cite[pp. 253--256]{Nathanson}. We define $f_n(s)$ by the multiple integral
\begin{align*}
sf_n(s):= \int\cdots \int_{\mathcal{R}_n(s)} \frac{dt_1 \cdots dt_n}{(t_1 \cdots t_n)t_n}
\end{align*}
where 
\begin{align*}
\mathcal{R}_n(s):= \left\{(t_1, \cdots, t_n)\in \R^n : \substack{0<t_n<\cdots < t_1<\frac{1}{s}, \ \
t_1+\cdots +t_n+2t_n>1,\\
t_1+\cdots +t_m+2t_m<1 \ \ \text{for } m<n \ \ \text{and } m\equiv n \pmod 2}\right\}.
\end{align*}
For $z\geq 2$ and $\epsilon\in(0,\frac{1}{200})$, under the assumption \eqref{linear sieve inequality}, $T_{k}(D,z)$ can be bounded as
\begin{align*}
T_{k}(D,z)<V(z)\left(f_k(s)+\epsilon e^{14-s} \right)
\end{align*}
where $D$ is any real number satisfying $D\geq z$ for $n$ odd and $D\geq z^2$ for $n$ even. Invoking the above bound into \eqref{Gz+} and \eqref{Gz-} and applying \cite[Theorem 9.4]{Nathanson}, we obtain
\begin{align*}
&G(z, \lambda^+) < V(z)\left(F(s)+ \epsilon e^{14-s}\right)\\
&G(z, \lambda^-) > V(z)\left(f(s)- \epsilon e^{14-s}\right),
\end{align*}
where $F(s)$ and $f(s)$ are defined in \eqref{F(s) and f(s)}. Finally inserting the above bound into \eqref{bound Sapz}, we can conclude our lemma.
\end{proof}		
Our next goal here is to set the multiplicative function $g$ and the finite set $\mathcal{Q}$ arrived in Lemma \ref{Jurkat-Richert} to bound the sieving functions appeared in Proposition \ref{lowerbound of r(N)}. We set the multiplicative function $g$ by $g(\mathfrak{p})=\frac{1}{\phi(\mathfrak{p})}$ for $p\in\mathcal{P}$ and zero otherwise. Then as $\pm1\pm i\notin\mathcal{P}$, we have $$ 0<g(\mathfrak{p})=\frac{1}{N(\mathfrak{p})-1}<1, \text{ for all } p\in\mathcal{P}.$$
In the next lemma, we have shown that the above $g$ satisfies the inequality \eqref{linear sieve inequality} in Lemma \ref{Jurkat-Richert}. 
	\begin{lemma}
For any $z>2$ and for any $\epsilon>0$, there exist $u_2(\epsilon)>0$ such that the function $g$ satisfies
		\begin{align}\label{g_linear}
				\prod_{\substack{u\leq N(\mathfrak{p})<z}}(1-g(\mathfrak{p}))^{-1}<(1+\epsilon)\frac{\log z}{\log u},
			\end{align}
			for all $u_2(\epsilon)\leq u<z.$
	\end{lemma}
\begin{proof}
We first split the product on the left hand side of \eqref{g_linear} as
	\begin{align}\label{Middle_inequality}
		\prod_{\substack{u\leq N(\mathfrak{p})<z}}(1-g(\mathfrak{p}))^{-1}&=\prod_{\substack{u\leq N(\mathfrak{p})<z}}\frac{(N(\mathfrak{p})-1)^2}{N(\mathfrak{p})(N(\mathfrak{p})-2)}\prod_{\substack{u\leq N(\mathfrak{p})<z}}\left(1-\frac{1}{N(\mathfrak{p})}\right)^{-1}.
	\end{align}
The first product can be written as	
$$\displaystyle\prod_{\mathfrak{p}}\frac{(N(\mathfrak{p})-1)^2}{N(\mathfrak{p})(N(\mathfrak{p})-2)}=\prod_{\mathfrak{p}}\left(1+\frac{1}{N(\mathfrak{p})(N(\mathfrak{p})-2)}\right)<\infty.$$
Thus for any $\epsilon>0,$ there exists $u_1(\epsilon)>0$ such that  for any $u_1(\epsilon)\leq u<z$
\begin{align}\label{2}
		\prod_{\substack{u\leq N(\mathfrak{p})<z}}\frac{(N(\mathfrak{p})-1)^2}{N(\mathfrak{p})(N(\mathfrak{p})-2)}<1+\epsilon/3.
	\end{align}
On the other hand, Lemma \ref{Lem:mertern} yields there exist $u_0(\epsilon)$ such that for any $u_0(\epsilon)\leq u<z$,
\begin{align*}
\prod_{\substack{u\leq N(\mathfrak{p})<z}}\left(1-\frac{1}{N(\mathfrak{p})}\right)^{-1}<(1+\epsilon/3)\frac{\log z}{\log u}.
\end{align*}
Finally, by setting $u_2(\epsilon)=\max\{u_0(\epsilon),u_1(\epsilon)\}$ and inserting \eqref{2} and \eqref{Pre_sieve_inequality} into \eqref{Middle_inequality}, we have 
	\begin{align*}
		\prod_{\substack{u\leq N(\mathfrak{p})<z}}(1-g(\mathfrak{p}))^{-1}<\left(1+\frac{\epsilon}{3}\right)^2 \frac{\log z}{\log u}<(1+\epsilon)\frac{\log z}{\log u},
	\end{align*}
for all $u_2(\epsilon)\leq u<z.$ This completes our lemma.	
\end{proof}		
We next fix the finite set $\mathcal{Q}$ by the set of all primes in $\mathcal{P}$ with $N(p)<u_2(\epsilon)$ and consider $Q=\displaystyle\prod_{\substack{\mathfrak{p}\\p\in\mathcal{Q}}}N(\mathfrak{p})$. Since $Q$ depends only on $\epsilon$, not on $\mathcal{N}$, therefore $Q$ can be bounded as 
\begin{align}\label{Bound on Q}
Q<\log N(\mathcal{N})
\end{align}
for sufficiently large $N(\mathcal{N})$. The next lemma provides an asymptotic estimate of $V(z)$ as defined in \eqref{Def of V(z)}.
\begin{lemma}\label{lemma_for_V(z)}
				Let $\mathcal{N}$ be an even Gaussian integer. Then for $z\geq 4$, we have 
				\begin{align*}
					V(z)=\frac{1}{2\pi}e^{-\gamma}\mathfrak{S}_1(\mathcal{N})\frac{1}{\log z}\left(1+O\left(\frac{1}{\log z}\right)\right),
				\end{align*} 
where $\mathfrak{S}_1(\mathcal{N})$ is defined in \eqref{Sing series}.
\end{lemma}
\begin{proof}
It follows from the definition \eqref{Def of V(z)} of $V(z)$ that
\begin{align}\label{V(z) cal}
V(z)
&=\prod_{\substack{2<N(\mathfrak{p})<z\\\mathfrak{p}\nmid \mathcal{N}}}\left(1-\frac{1}{N(\mathfrak{p})-1}\right)\nonumber\\
					&=\prod_{\substack{2<N(\mathfrak{p})<z}}\left(1-\frac{1}{N(\mathfrak{p})-1}\right)\prod_{\substack{N(\mathfrak{p})>2\\\mathfrak{p}\mid \mathcal{N}}}\left(1-\frac{1}{N(\mathfrak{p})-1}\right)^{-1}\prod_{\substack{N(\mathfrak{p})\geq z\\\mathfrak{p}\mid \mathcal{N}}}\left(1-\frac{1}{N(\mathfrak{p})-1}\right).
				\end{align}
The last product of the above equation can be approximated as	
\begin{align}\label{V(z)_1}
\prod_{\substack{N(\mathfrak{p})\geq z\\\mathfrak{p}\mid \mathcal{N}}}\left(1-\frac{1}{N(\mathfrak{p})-1}\right)=1+O\left(\frac{\log N(\mathcal{N})}{z}\right)
\end{align}
by applying the facts $1-x>e^{-2x}$ for $0<x<(\log 2)/2$ and $1-x<e^{-x}$ for all $x$.  
The first product in \eqref{V(z) cal} reduces to
	\begin{align}\label{V(z)_2}
		\prod_{\substack{2<N(\mathfrak{p})<z}}\left(1-\frac{1}{N(\mathfrak{p})-1}\right)
		&=2\prod_{\substack{N(\mathfrak{p})<z}}\left(1-\frac{1}{N(\mathfrak{p})}\right)\prod_{\substack{2<N(\mathfrak{p})<z}}\left(1-\frac{1}{N(\mathfrak{p})-1}\right)\left(1-\frac{1}{N(\mathfrak{p})}\right)^{-1}\nonumber\\
		&=2\prod_{\substack{N(\mathfrak{p})<z}}\left(1-\frac{1}{N(\mathfrak{p})}\right)\prod_{\substack{N(\mathfrak{p})>2}}\left(1-\frac{1}{(N(\mathfrak{p})-1)^2}\right)\prod_{\substack{N(\mathfrak{p})\geq z}}\left(1+\frac{1}{N(\mathfrak{p})(N(\mathfrak{p})-2)}\right)\nonumber\\
		&=2\prod_{\substack{N(\mathfrak{p})<z}}\left(1-\frac{1}{N(\mathfrak{p})}\right)\prod_{\substack{N(\mathfrak{p})>2}}\left(1-\frac{1}{(N(\mathfrak{p})-1)^2}\right)\left(1+O\left(\frac{1}{z}\right)\right),
	\end{align}
where the last step follows due to the inequality $1+x<e^x<1+2x$ for $0<x<\log 2$. It follows from \eqref{Merten_2} that 
\begin{align}\label{V(z)_3}
\prod_{\substack{N(\mathfrak{p})<z}}\left(1-\frac{1}{N(\mathfrak{p})}\right)=\frac{4e^{-\gamma}}{\pi \log z}\left(1+O\left(\frac{1}{\log z}\right)\right)
\end{align}
Inserting \eqref{V(z)_3} into \eqref{V(z)_2}, we can combine  \eqref{V(z) cal}, \eqref{V(z)_1} and \eqref{V(z)_2} to conclude that
	\begin{align*}
		V(z)=\frac{1}{2\pi}e^{-\gamma}\mathfrak{S}_1(\mathcal{N})\frac{1}{\log z}\left(1+O\left(\frac{1}{\log z}\right)\right).
	\end{align*}
This completes the proof of the lemma.	
\end{proof}	
\section{Lower bound of $S(\mathcal{A}, \mathcal{P}, z)$}\label{Estimation of S(A, P, z)}  
This section is concerned in estimating the lower bound of $S(\mathcal{A}, \mathcal{P},z)$, which arrived in the right hand side of \eqref{lowerbound of r(N)1} and for that our goal is to apply Lemma \ref{Jurkat-Richert}. To that end, we first need the asymptotic formula of the cardinality of  $\mathcal{A}_\mathfrak{d}$. For $g(\mathfrak{d}) = \frac{1}{\phi(\mathfrak{d})}$, it follows from \eqref{A_d} that the main term of $|\mathcal{A}_\mathfrak{d}|$ is $\frac{|\mathcal{A}|}{\phi(\mathfrak{d})}$. The error term $r(\mathfrak{d})$ is provided in the next lemma.
\begin{lemma}\label{Lem:A_d error}
The error term of $|\mathcal{A}_\mathfrak{d}|$ satisfies
$$r(\mathfrak{d})=\pi(N(\mathcal{N}); d, \mathcal{N})-4\frac{\pi(N(\mathcal{N}))}{\phi(\mathfrak{d})}+O(\log N(\mathcal{N})),$$ 
\end{lemma}
\begin{proof}
We can estimate $|\mathcal{A}_\mathfrak{d}|$ as 
	\begin{align}\label{Formula_A_q}
		|\mathcal{A}_\mathfrak{d}|=\sum_{\substack{n\in\mathcal{A}\\\mathfrak{d}|n}}1 =\sum_{\substack{\mathcal{N}-p\in\mathcal{A}\\ \mathcal{N}-p\equiv 0(\bmod d)}}1 =\sum_{\substack{p\in\mathcal{P}\\N(p)< N(\mathcal{N})\\ p\equiv \mathcal{N}(\bmod d)}}1 &=\sum_{\substack{N(p)< N(\mathcal{N})\\ p\equiv\mathcal{N}(\bmod d)}}1 +O(\omega(\mathcal{N} ))\nonumber\\
&=\pi(N(\mathcal{N}); d, \mathcal{N})+O(\log N(\mathcal{N})),
	\end{align}
	where in the last step we use the bound $\omega(\mathcal{N})\ll \log N(\mathcal{N}).$
Thus, applying \eqref{A card}, it follows from \eqref{A_d} that the error term can be approximated as
	\begin{align*}\label{r(d)}
		r(\mathfrak{d})
		&=|\mathcal{A}_\mathfrak{d}|-\frac{|\mathcal{A}|}{\phi(\mathfrak{d})}\nonumber\\
		&=\pi(N(\mathcal{N}); d, \mathcal{N})+O(\log N(\mathcal{N}))-4\frac{\pi(N(\mathcal{N}))-\omega(\mathcal{N})}{\phi(\mathfrak{d})}\nonumber\\
		&=\pi(N(\mathcal{N}); d,\mathcal{N})-4\frac{\pi(N(\mathcal{N}))}{\phi(\mathfrak{d})}+O(\log N(\mathcal{N})).
	\end{align*}
\end{proof} 
In order to apply Lemma~\ref{Jurkat-Richert}, we bound the error term of $S(\mathcal{A}, \mathcal{P}, z)$ 
in the following lemma. 
\begin{lemma}\label{Bound error}
Let $D=\displaystyle\frac{N(\mathcal{N})^{1/2}}{(\log N(\mathcal{N}))^{B_3+1}}$ for some  some $B_3>0$. Then we have the bound
$$\sum\limits_{\substack{\mathfrak{d}|P(z)\\N(\mathfrak{d})<DQ}}|r(\mathfrak{d})| \ll \frac{N(\mathcal{N})}{(\log N(\mathcal{N}))^3}.$$
\end{lemma}
\begin{proof}
It follows from \eqref{Bound on Q} that for sufficiently large $N(\mathcal{N})$,
\begin{equation}\label{Bound DQ}
DQ<\frac{N(\mathcal{N})^{1/2}}{(\log N(\mathcal{N}))^{B_3}}.
\end{equation}
Letting, 
\begin{align}\label{Delta}
\delta(N(\mathcal{N});d,\mathcal{N}):= \pi(N(\mathcal{N}); d,\mathcal{N})-4\frac{\pi(N(\mathcal{N}))}{\phi(\mathfrak{d})},
\end{align}
we apply Lemma \ref{B-V thm} with $A=3$ to obtain 

\begin{equation}\label{delta sum}
\sum_{\substack{\mathfrak{d}|P(z)\\4<N(d)<DQ}}|\delta(N(\mathcal{N});d,\mathcal{N})|
\ll \frac{N(\mathcal{N})}{(\log N(\mathcal{N}))^3}.
\end{equation}
Next, we observe the values of $\delta(N(\mathcal{N});d,\mathcal{N})$ for $1\leq N(d)\leq 4$ with $\mathfrak{d}|P(z)$. But for any $d\in\Z[i]$ with $\mathfrak{d}|P(z)$, the norm can not take the values $2, 3$ and $4$ since more explicitly, for  $N(d)=2$, we have $ d=\pm 1\pm i$, which is impossible as $\pm 1\pm i\notin\mathcal{P}$ implies the ideal $(\pm 1\pm i)\nmid P(z)$, for the case $N(d)=3$, writing $d=a_1+ib_1$, we have $a_1^2+b_1^2=3$, which is  again not possible as $a_1,b_1\in\mathbb{Z}$ and finally, for the case $N(d)=4$ we have $d=\pm2, \pm 2i$ which implies $(\pm 1\pm i)\mid \mathfrak{d}$, which is not possible as the ideals $(\pm 1\pm i)\nmid P(z)$. Therefore, the only possible case is $N(d)=1$, that is, $d$ is a unit in $\mathbb{Z}[i]$. But for any unit $u$ in $\Z[i]$, we have
	\begin{align*}
		\delta(N(\mathcal{N});u,\mathcal{N})&=\pi(N(\mathcal{N}); u, \mathcal{N})-4\frac{\pi(N(\mathcal{N}))}{\phi(u)}\\
		&=4\pi(N(\mathcal{N}))-4\pi(N(\mathcal{N}))\\
		&=0.
	\end{align*}
Therefore the sum \eqref{delta sum} for $1\leq N(d)\leq 4$ reduces to
	\begin{align}\label{deltaless4}
		&\sum_{\substack{\mathfrak{d}|P(z)\\N(d)\leq4}}|\delta(N(\mathcal{N});d,\mathcal{N})|=0.
	\end{align}
Finally, it follows from Lemma \ref{Lem:A_d error} that
	\begin{align*}
		\sum_{\substack{\mathfrak{d}|P(z)\\N(\mathfrak{d})<DQ}}|r(\mathfrak{d})|&=\sum_{\substack{\mathfrak{d}|P(z)\\N(\mathfrak{d})<DQ}}\delta(N(\mathcal{N});d,\mathcal{N})+
		O\left(\sum_{\substack{\mathfrak{d}|P(z)\\N(\mathfrak{d})<DQ}}\log N(\mathcal{N})\right)\nonumber\\
		&\ll \frac{N(\mathcal{N})}{(\log N(\mathcal{N}))^3}+\log N(\mathcal{N})\frac{N(\mathcal{N})^{1/2}}{(\log N(\mathcal{N}))^{B_3}}\nonumber\\
		&\ll \frac{N(\mathcal{N})}{(\log N(\mathcal{N}))^3},
	\end{align*}
where in the penultimate step we have used the bound \eqref{delta sum} and \eqref{deltaless4} in the first term and \eqref{Bound DQ} in the second term. This completes the proof of the lemma.
\end{proof}
In the following lemma we obtain the lower bound of $S(\mathcal{A}, \mathcal{P}, z)$.
\begin{lemma}\label{1st sum}
For $z=N(\mathcal{N})^{1/8}$, we have
\begin{align*}
S(\mathcal{A}, \mathcal{P}, z)>\frac{4V(z) N(\mathcal{N})}{\log N(\mathcal{N})}\left(\frac{e^{\gamma}}{2}\log 3+O(\epsilon)\right).
\end{align*}
\end{lemma} 
\begin{proof}
For $D=\displaystyle\frac{N(\mathcal{N})^{1/2}}{(\log N(\mathcal{N}))^{B_3+1}}$ and $z=N(\mathcal{N})^{1/8}$, we have
	\begin{align*}
		s=\frac{\log D}{\log z}
		=8\frac{\log\left(\frac{N(\mathcal{N})^{1/2}}{(\log N(\mathcal{N}))^{B_3+1}}\right)}{\log N(\mathcal{N})}
		=4-\frac{8(B_3+1)\log\log N(\mathcal{N})}{\log N(\mathcal{N})}.
	\end{align*}
Thus, for sufficiently large $N(\mathcal{N})$, we have $s\in[3,4]$. Hence the condition \eqref{F(s) and f(s)} of $f(s)$ readily yields
	\begin{align*}
		f(s)=\frac{2 e^\gamma \log(s-1)}{s}
		=\frac{e^\gamma}{2}\log 3+O\left(\frac{\log\log N(\mathcal{N})}{\log N(\mathcal{N})}\right)
		=\frac{e^\gamma}{2}\log 3+O(\epsilon).
	\end{align*}
Finally, we invoke the above estimate of $f(s)$, cardinality of $\mathcal{A}$ from \eqref{A card} with Lemma~\ref{PIT} and the bound of the error term from Lemma \ref{Bound error} together into \eqref{Lower bound of JR thm}, to obtain the lower bound of $S(\mathcal{A},\mathcal{P},z)$ as
\begin{align*}
S(\mathcal{A},\mathcal{P},z)&>\left(\frac{e^\gamma}{2}\log 3+O(\epsilon)\right)V(z)\frac{4 N(\mathcal{N})}{\log N(\mathcal{N})}\left(1+O\left(\frac{1}{\log N(\mathcal{N})}\right)\right)+O\left(\frac{N(\mathcal{N})}{(\log N(\mathcal{N}))^3}\right)\\
&=\frac{4V(z) N(\mathcal{N})}{\log N(\mathcal{N})}\left[\frac{e^\gamma}{2}\log 3+O(\epsilon)+O\left(\frac{1}{V(z)(\log N(\mathcal{N}))^2}\right)\right]\\
&=\frac{4V(z) N(\mathcal{N})}{\log N(\mathcal{N})}\left(\frac{e^\gamma}{2}\log 3+O(\epsilon)\right),
\end{align*}
where in the final step we can ignore the last term using the estimate of $V(z)$ from Lemma \ref{lemma_for_V(z)}.  
\end{proof}

\section{Upper bound for the average sum of $S(\mathcal{A}_\mathfrak{q}, \mathcal{P}, z)$ }\label{Estimation of S(A_q, P, z)}
In this section, our main concern is to find an upper bound of $\displaystyle\sum_{z\leq  N(\mf{q})<y}S(\mathcal{A}_\mathfrak{q}, \mathcal{P}, z)$, which has arrived in the right hand side of \eqref{lowerbound of r(N)1}. For $\mathcal{A}_\mathfrak{q}=\{n\in\mathcal{A}:\mathfrak{q}\mid \mathfrak{n}\}$ as appeared in Lemma \ref{T_1 sum}, let $r_\mathfrak{q}(\mathfrak{d})$ denotes the error term of $|(\mathcal{A}_ \mathfrak{q})_\mathfrak{d}|$. We fix $D=\displaystyle\frac{N(\mathcal{N})^{1/2}}{(\log N(\mathcal{N}))^{B_4+1}}$ for some $B_4>0$, $y= N(\mathcal{N})^{\frac{1}{3}}$ and $z= N(\mathcal{N})^{\frac{1}{8}}$. Then for $N(\mf{q})<y$, we have $D_\mf{q}=\frac{D}{N(\mf{q})} \geq \frac{D}{y} \geq z$. Thus, with the above choice of $D_\mf{q}$ along with $N(\mf{q})<y$, it follows from \eqref{Upper bound of JR thm} of Lemma \ref{Jurkat-Richert}  that we can bound $S(\mathcal{A}_\mathfrak{q}, \mathcal{P}, z)$ as
\begin{align*}
	S(\mathcal{A}_\mathfrak{q}, \mathcal{P}, z)<(F(s_\mf{q})+\epsilon e^{14-s_\mf{q}})V(z)|\mathcal{A}_q|+\sum_{\substack{\mf{d}|P(z)\\N(\mf{d})<D_\mf{q}Q}}|r_{\mf{q}}(\mf{d})|,
\end{align*}
where $s_q=\frac{D_\mf{q}}{\log z}$. The definition of $\mathcal{A}_\mathfrak{q}$, defined in Lemma \ref{T_1 sum}, yields that we may assume $N(q,\mathcal{N})=1$, since for $\mathcal{N}-p \in \mathcal{A}_\mathfrak{q}$, we have the prime $q$ divides $\mathcal{N}-p$, but $q\mid \mathcal{N}$ implies $q\mid p$, which is not possible as $p\nmid \mathcal{N}$. Therefore, summing over $z\leq N(\mf{q})<y$ on the both side of the above equation, $\displaystyle\sum_{z\leq  N(\mf{q})<y}S(\mathcal{A}_\mathfrak{q}, \mathcal{P}, z)$ can be bounded as
\begin{align}\label{JR on S_2}
\sum_{z\leq  N(\mf{q})<y}S(\mathcal{A}_\mathfrak{q}, \mathcal{P}, z) <V(z)\sum_{\substack{z\leq N(\mf{q})<y\\N(q,\mathcal{N})=1}}(F(s_\mf{q})+\epsilon e^{14})|\mathcal{A}_q|+\sum_{\substack{z\leq N(\mf{q})<y\\N(q, \mathcal{N})=1}}\sum_{\substack{\mf{d}|P(z)\\N(\mf{d})<D_\mf{q}Q}}|r_{\mf{q}}(\mf{d})|.
\end{align}  
We next bound the error term and estimate the main term of the above equation. In order to bound the error term, we need the following bound of the partial sum of $\frac{1}{\phi(\mf{n})}$, where $\mf{n}$ is an ideal in $\Z[i]$.
\begin{lemma}\label{Lem:reciprocal_sum_phi}
For any $x>0$, we have 
    \begin{align}\label{reciprocal_sum_phi}
        \sum_{N(\mf{n})<x}\frac{1}{\phi(\mf{n})}\ll \log x.
    \end{align}
\end{lemma}
\begin{proof}
The definition \eqref{phin} of $\phi(\mf{n})$ yields
\begin{align*}
    \frac{1}{\phi(\mf{n})}=\frac{1}{N(\mf{n})}\prod_{\mf{p}\mid \mf{n}}\left(1-\frac{1}{N(\mf{p})}\right)^{-1}
    =\frac{1}{N(\mf{n})}\sum_{\substack{\mf{d}\subset \mathbb{Z}[i]\\\mf{d}^*\mid\mf{n}}}\frac{1}{N(\mf{d})},
\end{align*}
where $\mf{d}^*$ denotes the square-free part of $\mf{d}$. Thus the following partial sum can be written as 
\begin{align*}
    \sum_{N(\mf{n})<x}\frac{1}{\phi(\mf{n})}&=\sum_{N(\mf{n})<x}\frac{1}{N(\mf{n})}\sum_{\substack{\mf{d}\subset \mathbb{Z}[i]\\\mf{d}^*\mid\mf{n}}}\frac{1}{N(\mf{d})}\\
    &=\sum_{\substack{\mf{d}\subset \mathbb{Z}[i]}}\frac{1}{N(\mf{d})}\sum_{\substack{N(\mf{n})<x\\\mf{d}^*\mid\mf{n}}}\frac{1}{N(\mf{n})}\\
    &\leq\sum_{\substack{\mf{d}\subset \mathbb{Z}[i]}}\frac{1}{N(\mf{d}\mf{d}^*)}\sum_{\substack{N(\mf{m})<x/N(\mf{d}^*)}}\frac{1}{N(\mf{m})},
\end{align*}
Here the first sum is finite since for a prime ideal $\mf{p}\mid\mf{d}\mf{d}^*$, we have $\mf{p}^2\mid \mf{d}\mf{d}^*$ and the second sum is bounded by $\log x$. 
Thus we can conclude
\begin{align*}
\sum_{N(\mf{n})<x}\frac{1}{\phi(\mf{n})} \ll \log x.
\end{align*}
This completes the proof of the lemma.
\end{proof}
The next lemma bounds the error term of $\displaystyle\sum_{z\leq  N(\mf{q})<y}S(\mathcal{A}_\mathfrak{q}, \mathcal{P}, z)$ given in \eqref{JR on S_2}.
\begin{lemma}\label{Error_S_2} 
For $z>4$ and $y<DQ$, we have
\begin{align*}
\sum_{\substack{z\leq N(\mf{q})<y\\N(q, \mathcal{N})=1}}\sum_{\substack{\mf{d}|P(z)\\N(\mf{d})<D_\mf{q}Q}}|r_{\mf{q}}(\mf{d})|\ll \frac{N(\mathcal{N})}{(\log N(\mathcal{N}))^3}.
\end{align*}
\end{lemma}

\begin{proof}	
Let $\mathfrak{d}|P(z)$ and $z\leq N(\mathfrak{q})<y$. Then  we can write the error term of $|(\mathcal{A}_ \mathfrak{q})_\mathfrak{d}|$ as
\begin{align*}
r_\mathfrak{q}(\mathfrak{d})=|(\mathcal{A}_ \mathfrak{q})_\mathfrak{d}|-\sum_{n\in\mathcal{A}_\mathfrak{q}}g(\mathfrak{d})
=|\mathcal{A}_{\mathfrak{q}\mathfrak{d}}|-\frac{|\mathcal{A}_q|}{\phi(\mathfrak{d})}
=|\mathcal{A}_{\mathfrak{q}\mathfrak{d}}|-\frac{|\mathcal{A}|}{\phi(\mathfrak{q}\mathfrak{d})}+\frac{|\mathcal{A}|}{\phi(\mathfrak{q}\mathfrak{d})}-\frac{|\mathcal{A}_\mathfrak{d}|}{\phi(\mathfrak{d})}
=r(\mathfrak{q}\mathfrak{d})-\frac{r(\mathfrak{q})}{\phi(\mathfrak{d})}.
\end{align*}

Thus 
\begin{align*}
\sum_{\substack{z\leq N(\mf{q})<y\\N(q, \mathcal{N})=1}}\sum_{\substack{\mf{d}|P(z)\\N(\mf{d})<D_{\mf{q}} Q}}|r_{\mf{q}}(\mf{d})|
&\leq \sum_{\substack{z\leq N(\mf{q})<y\\N(q, \mathcal{N})=1}}\sum_{\substack{\mf{d}|P(z)\\N(\mf{d})<D_{\mf{q}} Q}}|r(\mf{q}\mf{d})|+ \sum_{\substack{z\leq N(\mf{q})<y\\N(q, \mathcal{N})=1}}|r(\mf{q})|\sum_{\substack{\mf{d}|P(z)\\N(\mf{d})<D_{\mf{q}} Q}}\frac{1}{\phi(\mf{d})}\\
&\leq \sum_{\substack{4<N(\mf{d})<QD\\N(d ,\mathcal{N})=1}}|r(\mf{d})|+ \sum_{\substack{z\leq N(\mf{q})<y\\N(q, \mathcal{N})=1}}|r(\mf{q})|\sum_{\substack{\mf{d}|P(z)\\N(\mf{d})<D_q Q}}\frac{1}{\phi(\mf{d})}.
\end{align*}
Now for $z>4$, $y<DQ$ and applying the bound \eqref{reciprocal_sum_phi} in the last term, we can combine the above two terms into a single sum to obtain the bound as
\begin{align*}
\sum_{\substack{z\leq N(\mf{q})<y\\N(q, \mathcal{N})=1}}\sum_{\substack{\mf{d}|P(z)\\N(\mf{d})<D_{\mf{q}} Q}}|r_{\mf{q}}(\mf{d})|
 \ll \log N(\mathcal{N})\sum_{\substack{4<N(\mf{d})<QD\\N(d ,\mathcal{N})=1}}|r(\mf{d})| 
\ll \frac{N(\mathcal{N})}{(\log N(\mathcal{N}))^3},
\end{align*}
where the final inequality follows from the similar argument as in Lemma \ref{Bound error} and applying Lemma~\ref{B-V thm} with $A=4$.
\end{proof}
In the next lemma, we estimate the main term of $\displaystyle\sum_{z\leq  N(\mf{q})<y}S(\mathcal{A}_\mathfrak{q}, \mathcal{P}, z)$.
\begin{lemma}\label{S_2_Almost_final}
For any $\epsilon\in(0,\frac{1}{200})$, we have 
\begin{align*}
\sum_{\substack{z\leq N(\mf{q})<y\\N(q,\mathcal{N})=1}}(F(s_\mf{q})+\epsilon e^{14})|\mathcal{A}_q|< N(\mathcal{N})\Bigg[e^\gamma\sum_{\substack{z\leq N(\mf{q})<y\\N(q, \mathcal{N})=1}}\frac{1/N(\mf{q})}{\log\frac{N(\mathcal{N})^{1/2}}{N(\mf{q})}}+O\left(\sum_{\substack{z\leq N(\mf{q})<y\\N(q, \mathcal{N})=1}}\frac{1/N(\mf{q})^2}{\log\frac{N(\mathcal{N})^{1/2}}{N(\mf{q})}}\right)
+O\left(\frac{\epsilon}{\log N(\mathcal{N})}\right)\Bigg].
\end{align*}
\end{lemma}
\begin{proof}
For $D=\displaystyle\frac{N(\mathcal{N})^{1/2}}{(\log N(\mathcal{N}))^{B_4+1}}$  with $B_4>0$, we can reduce $s_q$ as
\begin{align*}
s_q =\frac{\log\frac{N(\mathcal{N})^{1/2}}{N(\mf{q}) (\log N(\mathcal{N}))^{B_4+1}}}{\log z} =\frac{\log\frac{N(\mathcal{N})^{1/2}}{N(\mf{q})}}{\log z}-\frac{(B_4+1)\log \log N(\mathcal{N})}{\log z}.
	\end{align*}
Thus for $z <N(\mf{q})\leq y$ and for sufficiently large $N(\mathcal{N})$, we have $1<s_q \leq 3$.
Therefore, the definition \eqref{F(s) and f(s)} yields
\begin{align}\label{F of S_2}
F(s_q) = \frac{e^\gamma \log N(\mathcal{N})}{4\log\frac{N(\mathcal{N})^{1/2}}{N(\mf{q})}}+O\left(\frac{\log\log N(\mathcal{N})}{\log N(\mathcal{N})}\right)
=\frac{e^\gamma \log N(\mathcal{N})}{4\log\frac{N(\mathcal{N})^{1/2}}{N(\mf{q})}}+O(\epsilon).
	\end{align}
We also have from \eqref{Formula_A_q} and \eqref{Delta} that 
	\begin{align}\label{Card of Aq}
		|A_{\mf{q}}|&=\pi(N(\mathcal{N}); q, \mathcal{N})+O(\log N(\mathcal{N}))\nonumber\\
		&=4\frac{\pi(N(\mathcal{N}))}{\phi(\mathfrak{q})}+\delta(N(\mathcal{N});q,\mathcal{N})+O(\log N(\mathcal{N}))\nonumber\\
		&=\frac{4 N(\mathcal{N})}{\phi(\mf{q})\log N(\mathcal{N})}\left(1+O\left(\frac{1}{\log N(\mathcal{N})}\right)\right)+\delta(N(\mathcal{N});q,\mathcal{N}),
	\end{align} 
where in the last step we have applied Lemma~\ref{PIT}. Applying \eqref{Upper bound of JR thm} of Lemma \ref{Jurkat-Richert} on $S(\mathcal{A}_\mathfrak{q}, \mathcal{P}, z)$, we now estimate the main term in \eqref{JR on S_2} and for that we first invoke \eqref{F of S_2} and \eqref{Card of Aq} into \eqref{Upper bound of JR thm} to write the main term as
	\begin{align}\label{Sum_over_F(s_q)}
		&\sum_{\substack{z\leq N(\mf{q})<y\\N(q,\mathcal{N})=1}}(F(s_\mf{q})+\epsilon e^{14})|\mathcal{A}_{\mf{q}}|\nonumber\\
		&=\sum_{\substack{z\leq N(\mf{q})<y\\N(q,\mathcal{N})=1}}\left[\left(\frac{e^\gamma \log N(\mathcal{N})}{4\log\frac{N(\mathcal{N})^{1/2}}{N(\mf{q})}} +O(\epsilon)\right)\frac{4 N(\mathcal{N})}{\phi(\mf{q})\log N(\mathcal{N})}\left(1+O\left(\frac{1}{\log N(\mathcal{N})}\right)\right)\right]\nonumber\\
		&+\sum_{\substack{z\leq N(\mf{q})<y\\N(q,\mathcal{N})=1}}\left(\frac{e^\gamma \log N(\mathcal{N})}{4\log\frac{N(\mathcal{N})^{1/2}}{N(\mf{q})}} +O(\epsilon)\right)\delta(N(\mathcal{N});q,\mathcal{N})\nonumber\\
		&=e^\gamma N(\mathcal{N})\sum_{\substack{z\leq N(\mf{q})<y\\N(q,\mathcal{N})=1}}\frac{1}{\phi(\mf{q})\log\frac{N(\mathcal{N})^{1/2}}{N(\mf{q})}}+O\left(\frac{\epsilon N(\mathcal{N})}{\log N(\mathcal{N})}\right)\sum_{\substack{z\leq N(\mf{q})<y\\N(q,\mathcal{N})=1}}\frac{1}{\phi(\mf{q})}\nonumber\\
		&+O\left(\frac{N(\mathcal{N})}{\log N(\mathcal{N})}\right)\sum_{\substack{z\leq N(\mf{q})<y\\N(q,\mathcal{N})=1}}\frac{1}{\phi(q\mf{q}\log\frac{N(\mathcal{N})^{1/2}}{N(\mf{q})}}+O\left(\frac{\epsilon N(\mathcal{N})}{(\log N(\mathcal{N}))^2}\right)\sum_{\substack{z\leq N(\mf{q})<y\\N(q,\mathcal{N})=1}}\frac{1}{\phi(\mf{q})}+O\left(\frac{N(\mathcal{N})}{(\log N(\mathcal{N}))^3}\right),
	\end{align}
where in the last sum of the penultimate step, we mainly used the fact that $F(s_{\mf{q}})$ is bounded for $1<s_{\mf{q}}\leq 3$ and applied Lemma~\ref{B-V thm} to obtain the final big-oh term in the last step. 

We next reduce the sums arrived in the above equation. It follows from Lemma~\ref{Merten's theorem_in_z[i]} that for sufficiently large $N(\mathcal{N})$, we can write
	\begin{align}\label{Reci_phi_over_prime}
		\sum_{\substack{z\leq N(\mf{q})<y\\N(q, \mathcal{N})=1}}\frac{1}{\phi(\mf{q})}&=\sum_{\substack{z\leq N(\mf{q})<y\\N(q, \mathcal{N})=1}}\frac{1}{N(\mf{q})-1}
		\ll \sum_{\substack{z\leq N(\mf{q})<y}}\frac{1}{N(\mf{q})}
		=\log\log y-\log\log z + O\left(\frac{1}{\log z}\right)
		=O(1).
	\end{align} 
Thus, the following sum can be bounded as
\begin{align}\label{in_between_main_term}
\sum_{\substack{z\leq N(\mf{q})<y\\N(q, \mathcal{N})=1}}\frac{1}{\phi(\mf{q})\log\frac{N(\mathcal{N})^{1/2}}{N(\mf{q})}}
		\ll\frac{1}{\log N(\mathcal{N})}\sum_{\substack{z\leq N(\mf{q})<y\\N(q, \mathcal{N})=1}}\frac{1}{\phi(\mf{q})}
		\ll\frac{1}{\log N(\mathcal{N})}.
	\end{align}
Inserting \eqref{Reci_phi_over_prime} and \eqref{in_between_main_term} into \eqref{Sum_over_F(s_q)} and using the fact that
\begin{align*}
		\frac{1}{\phi(\mf{q})}=\frac{1}{N(\mf{q})-1}=\frac{1}{N(\mf{q})}+O\left(\frac{1}{N(\mf{q})^2}\right),
	\end{align*}
the main term can be reduced to	
	\begin{align*}
		&\sum_{\substack{z\leq N(\mf{q})<y\\N(q,\mathcal{N})=1}}(F(s_\mf{q})+\epsilon e^{14})|\mathcal{A}_q|\nonumber\\
		&=e^\gamma N(\mathcal{N})\sum_{\substack{z\leq N(\mf{q})<y\\N(q, \mathcal{N})=1}}\frac{1}{\phi(\mf{q})\log\frac{N(\mathcal{N})^{1/2}}{N(\mf{q})}}+O\left(\frac{\epsilon N(\mathcal{N})}{\log N(\mathcal{N})}\right)\nonumber\\
		&=e^\gamma N(\mathcal{N})\sum_{\substack{z\leq N(\mf{q})<y\\N(q, \mathcal{N})=1}}\frac{1}{N(\mf{q})\log\frac{N(\mathcal{N})^{1/2}}{N(\mf{q})}}+O\left(N(\mathcal{N})\sum_{\substack{z\leq N(\mf{q})<y\\N(q, \mathcal{N})=1}}\frac{1}{N(\mf{q})^2\log\frac{N(\mathcal{N})^{1/2}}{N(\mf{q})}}\right)+O\left(\frac{\epsilon N(\mathcal{N})}{\log N(\mathcal{N})}\right).
	\end{align*}
This completes the proof of the lemma.
\end{proof}
In the following lemma, we obtain the upper bound of the sum $\displaystyle\sum_{z\leq  N(\mf{q})<y}S(\mathcal{A}_\mathfrak{q}, \mathcal{P}, z)$.
\begin{lemma}\label{Estimate_S_2}
For $y= N(\mathcal{N})^{\frac{1}{3}}$, $z= N(\mathcal{N})^{\frac{1}{8}}$ and any $\epsilon\in(0,\frac{1}{200})$, the bound
	\begin{align*}
		\sum_{z\leq  N(\mf{q})<y}S(\mathcal{A}_\mathfrak{q}, \mathcal{P}, z)<\frac{4V(z)N(\mathcal{N})}{\log N(\mathcal{N})}\left(\frac{e^\gamma \log 6}{2}+O(\epsilon)\right)
	\end{align*}
holds.	
\end{lemma}
\begin{proof}
For any natural number $k$, we define an arithmetic function
\begin{align*}
a_k := \begin{cases}
			\frac{1}{N(\mf{q})} & \text{ if } k=N(\mf{q}) \text{ for some prime ideal } \mf{q} \in \mathbb{Z}[i],\\
		\ \	0 & \text{ otherwise. }
		\end{cases}
\end{align*} 
For any $t>0$, let $S(t)$ denotes the partial sum of $a_k$ given by
	\begin{align*}
		S(t):=\sum_{\substack{\mf{q}\\N(\mf{q})<t}}\frac{1}{N(\mf{q})}=\sum_{1\leq k<t}a_k.
	\end{align*}
Converting the partial sum into Riemann-Stieltjes integral, we have
	\begin{align}\label{Abel_summation}
		\sum_{\substack{z\leq N(\mf{q})<y\\N(q, \mathcal{N})=1}}\frac{1}{N(\mf{q})\log\frac{N(\mathcal{N})^{1/2}}{N(\mf{q})}}
		&=\int_{z}^{y}\frac{1}{\log\left(\frac{N(\mathcal{N})^{1/2}}{t}\right)}dS(t)\nonumber\\
		&=\int_{z}^{y}\frac{1}{\log\left(\frac{N(\mathcal{N})^{1/2}}{t}\right)}d\left(\log\log t +B+O\left(\frac{1}{\log t}\right)\right)\nonumber\\
		&=\int_{z}^{y}\frac{1}{\log\left(\frac{N(\mathcal{N})^{1/2}}{t}\right)}d(\log\log t)+O\left(\int_{z}^{y}\frac{1}{t(\log t)^2 \log\left(\frac{N(\mathcal{N})^{1/2}}{t}\right)}dt\right),
	\end{align}
where the penultimate step follows from Lemma~\ref{Merten's theorem_in_z[i]}. Now making the change of variable $t=N(\mathcal{N})^{x}$ and inserting the values of $y$ and $z$, the first integral reduces to 
\begin{align}\label{6.1 integral}
\int_{z}^{y}\frac{1}{\log\left(\frac{N(\mathcal{N})^{1/2}}{t}\right)}d(\log\log t) 
= \frac{1}{\log N(\mathcal{N})}\int_{1/8}^{1/3}\frac{dx}{x(1/2-x)}
=\frac{2\log 6}{\log N(\mathcal{N})}	
\end{align}
and the second integral can be bounded as
\begin{align}\label{6.2 integral}
\int_{z}^{y}\frac{1}{t(\log t)^2 \log\left(\frac{N(\mathcal{N})^{1/2}}{t}\right)}dt\ll \frac{1}{(\log N(N))^3}\int_{z}^{y} \frac{dt}{t}  \ll \frac{1}{(\log N(N))^2}.
\end{align}
Thus by invoking \eqref{6.1 integral} and \eqref{6.2 integral} into \eqref{Abel_summation}, we obtain 
	\begin{align}\label{1st_sum_in_S_2}
		\sum_{\substack{z\leq N(\mf{q})<y\\N(q, \mathcal{N})=1}}\frac{1}{N(\mf{q})\log\frac{N(\mathcal{N})^{1/2}}{N(\mf{q})}}=\frac{2\log 6}{\log N(\mathcal{N})}+O\left(\frac{1}{(\log N(N))^2}\right).
	\end{align}
Also the following sum can be bounded as
\begin{equation}\label{2nd_sum_in_S_2}
\sum_{\substack{z\leq N(\mf{q})<y\\N(q, \mathcal{N})=1}}\frac{1}{N(\mf{q})^2\log\frac{N(\mathcal{N})^{1/2}}{N(\mf{q})}}
\ll \frac{1}{\log N(\mathcal{N})}\sum_{\substack{z\leq N(\mf{q})<y}}\frac{1}{N(\mf{q})^2}
\ll \frac{1}{z\log N(\mathcal{N})}\sum_{\substack{z\leq N(\mf{q})<y}}\frac{1}{N(\mf{q})}
\ll\frac{1}{z\log N(\mathcal{N})},
\end{equation}
where the last step holds due to \eqref{Reci_phi_over_prime}. Finally after applying \eqref{1st_sum_in_S_2} and \eqref{2nd_sum_in_S_2} together in Lemma~\ref{S_2_Almost_final}, we combine \eqref{JR on S_2}, Lemma \ref{Error_S_2} and Lemma~\ref{S_2_Almost_final} to conclude
	\begin{align*}
\sum\limits_{\substack{z\leq N(\mf{q})<y}}\hspace{-5pt}S(\mathcal{A}_\mathfrak{q}, \mathcal{P},z)
&<V(z)N(\mathcal{N})\left[e^\gamma\left(\frac{2\log 6}{\log N(\mathcal{N})}+O\left(\frac{1}{(\log N(N))^2}\right)\right)+O\left(\frac{1}{z\log N(\mathcal{N})}\right)+O\left(\frac{\epsilon}{\log N(\mathcal{N})}\right)\right]\\
&\qquad +O\left(\frac{N(\mathcal{N})}{(\log N(\mathcal{N}))^3}\right)\\
		&=\frac{4V(z)N(\mathcal{N})}{\log N(\mathcal{N})}\left[\frac{e^\gamma \log 6}{2}+O\left(\frac{1}{\log N(\mathcal{N})}\right)+O\left(\frac{1}{z}\right)+O(\epsilon)+O\left(\frac{1}{V(z)(\log N(\mathcal{N}))^2}\right)\right]\\
		&=\frac{4V(z)N(\mathcal{N})}{\log N(\mathcal{N})}\left(\frac{e^\gamma \log 6}{2}+O(\epsilon)\right),
	\end{align*}
where the final step follows by inserting the value $z=N(\mathcal{N})^{1/8}$ and the expression of $V(z)$ from Lemma~\ref{lemma_for_V(z)}. This completes the proof of the lemma.
\end{proof}
		
\section{Upper bound of $S(\mathcal{B},\mathcal{P},y)$}\label{Estimation of S(B, P, y)}
In this section, we find an upper bound of the sieving function $S(\mathcal{B},\mathcal{P},y)$, where the set $\mathcal{B}$ is defined in \eqref{Def of B} as 
\begin{align*}
	\mathcal{B}=\{\mathcal{N}-p_1 p_2 p_3:&z\leq N(p_1)<y\leq N(p_2)\leq N(p_3),N(p_1 p_2 p_3)<4N(\mathcal{N}),N(p_1 p_2 p_3, \mathcal{N})=1\}.
\end{align*}
We divide the range of $N(p_1)$ into disjoint intervals as 
\begin{align*}
 [z,y)=\bigcup_{k=0}^{r}\left[\ell_k, \ell_{k+1}\right),
\end{align*}
where $\ell_k = z(1+\epsilon)^k$ for some $\epsilon>0$. Then, for every $0\leq k\leq r$, it follows that $z(1+\epsilon)^{1+k}\leq y$, which implies
\begin{align}\label{bournd of r}
k\leq\frac{\log(y/z)}{\log (1+\epsilon)}\ll \frac{\log N(\mathcal{N})}{\epsilon}.
\end{align}
We next define the set
\begin{align}\label{Def of B^k}
\mathcal{B}^{(k)}:=\{\mathcal{N}-p_1 p_2 p_3:&z\leq N(p_1)<y\leq N(p_2)\leq N(p_3),\ell_k\leq N(p_1)<\ell_{k+1},\nonumber\\
& \ell_kN(p_2 p_3)<4N(\mathcal{N}), N(p_2 p_3, \mathcal{N})=1\}
\end{align}
and denote $\tilde{\mathcal{B}}:=\bigcup_{k=0}^{r}\mathcal{B}^{(k)}$. Clearly, it follows from the above definitions that $\mathcal{B} \subseteq \tilde{\mathcal{B}}$ and since the sets $\mathcal{B}^{(k)}$ are pairwise disjoint, we have 
\begin{align}\label{Inequality_B_B'}
	S(\mathcal{B},P,y)\leq S(\tilde{\mathcal{B}}, P,y)=\sum_{k=0}^{r}S(\mathcal{B}^{(k)}, P,y).
\end{align}
Thus the above equation implies that the upper bound of  $S(\mathcal{B}^{(k)}, P,y)$ will lead to the upper bound of  
$S(\tilde{\mathcal{B}}, P,y)$ and the resulting bound will also work for $S(\mathcal{B},P,y)$. We next bound the error term of $S(\mathcal{B}^{(k)}, P,y)$ and for that we need the following lemma, which can be considered as an analogue of \cite[Theorem 10.7]{Nathanson} in $\mathbb{Z}[i]$ set up. 
	\begin{lemma}\label{errorbound}
Let $A, X, Y, Z$ be positive real numbers with $X>(\log Y)^{2A}$ and $D^*=\frac{(XY)^{1/2}}{(\log Y)^A}$. Then for any complex-valued function $a(n)$ on $\Z[i]$, we have
\begin{align*}
		\sum_{\substack{N(\mf{d})<D^*\\\mf{d}|P(y)}}\max_{N(a,d)=1}\left|\sum_{N(n)<X}\sum_{\substack{Z\leq N(p)<Y\\np\equiv a(\bmod d)}}a(n)-\frac{1}{\phi(\mf{d})}\sum_{N(n)<X}\sum_{\substack{Z\leq N(p)<Y\\N(np,d)=1}}a(n)\right|\ll_A \frac{XY(\log XY)^2}{(\log Y)^A}.
\end{align*}
\end{lemma}
The proof of the lemma goes almost along the similar direction as in \cite[pp. 292]{Nathanson} with few modifications. The main ingredient to prove the lemma is the large sieve inequality in $\mathbb{Z}[i]$, which was originally established by Huxley \cite[Theorem 4]{Huxley1} for general number field. In $\mathbb{Z}[i]$, the inequality [cf. \cite[Equation~(3)]{Arpit}] precisely states that for any $L,M\geq 1$ and for any complex-valued function $b(n)$ on $\mathbb{Z}[i]$, we have
	\begin{align*}
		\sum_{\substack{N(\mf{d})<L}}\frac{N(\mf{d})}{\phi(\mf{d})}\sideset{}{^*}\sum_{\substack{\chi(\bmod d)}}\left|\sum_{N(n)\leq M}b(n)\chi(n)\right|^2\ll (L^2 + M)\sum_{N(n)\leq M}|b(n)|^2,
	\end{align*}
where $\sideset{}{^*}\sum$ denotes the sum runs over primitive characters modulo $d$.
Let $\tilde{R}$ denotes the error term of $S(\tilde{\mathcal{B}}, P,y)$. In the following lemma we bound $\tilde{R}$.
\begin{lemma}\label{lemma_for_R_tilde}
Let $z=N(\mathcal{N})^{1/8}$ and $y=N(\mathcal{N})^{1/3}$. Then for $D:=\frac{N(\mathcal{N})^{1/2}}{(\log N(\mathcal{N}))^{7}}$, we have 
\begin{align*}
\tilde{R} \ll \frac{N(\mathcal{N})}{\epsilon(\log N(\mathcal{N}))^3}.
\end{align*}
\end{lemma}
\begin{proof}
Letting $g(\mathfrak{d}) = \frac{1}{\phi(\mathfrak{d})}$, \eqref{A_d} implies that the main term of $|\mathcal{B}^{(k)}_{\mf{d}}|$ is given by $\frac{|\mathcal{B}^{(k)}|}{\phi(\mf{d})}$. For $r^{(k)}_{\mf{d}}$ denoting the error term of $|\mathcal{B}^{(k)}_{\mf{d}}|$, it follows from the definition \eqref{Def of B^k} of $\mathcal{B}^{(k)}$ that
	\begin{align*}
r^{(k)}_{\mf{d}}
&=|\mathcal{B}^{(k)}_\mf{d}|-\frac{|\mathcal{B}^{(k)}|}{\phi(\mf{d})}
=\sum_{\substack{z\leq N(p_1)<y\leq N(p_2)\leq N(p_3)\\\ell_k\leq N(p_1)<\ell_{k+1}\\\ell_k N(p_2 p_3)<4N(\mathcal{N})\\N(p_2 p_3, \mathcal{N})=1\\p_1 p_2 p_3\equiv\mathcal{N}(\bmod d)}}1-\frac{1}{\phi(\mf{d})}\sum_{\substack{z\leq N(p_1)<y\leq N(p_2)\leq N(p_3)\\\ell_k\leq N(p_1)<\ell_{k+1}\\\ell_k N(p_2 p_3)<4N(\mathcal{N})\\N(p_2 p_3, \mathcal{N})=1}}1\\
&=\sum_{\substack{y\leq N(p_2)\leq N(p_3)\\N(p_2 p_3)<4N(\mathcal{N})/\ell_k\\N(p_2 p_3,\mathcal{N})=1}} \ \sum_{\substack{\ell_k\leq N(p_1)<\ell_{k+1}\\p_1 p_2 p_3\equiv \mathcal{N}(\bmod d)}}1
-\frac{1}{\phi(\mf{d})}\sum_{\substack{y\leq N(p_2)\leq N(p_3)\\N(p_2 p_3)<4N(\mathcal{N})/\ell_k\\N(p_2 p_3,\mathcal{N})=1}}\sum_{\substack{\ell_k\leq N(p_1)<\ell_{k+1}\\ N(p_1,d)=1}}1	
- \frac{1}{\phi(\mf{d})}\sum_{\substack{y\leq N(p_2)\leq N(p_3)\\\ell_k\leq N(p_1)<\ell_{k+1}\\ \ell_k N(p_2 p_3)<4N(\mathcal{N})\\N(p_2 p_3,\mathcal{N})=1\\N(p_1,d)>1}}1	
	\end{align*}
The last sum of the above equation can be bounded as 
\begin{align*}
\frac{1}{\phi(\mf{d})}\sum_{\substack{y\leq N(p_2)\leq N(p_3)\\\ell_k\leq N(p_1)<\ell_{k+1}\\ \ell_k N(p_2 p_3)<4N(\mathcal{N})\\N(p_2 p_3,\mathcal{N})=1\\N(p_1,d)>1}}1
\leq \frac{1}{\phi(\mf{d})}\sum_{\substack{ N(p_1)\geq z\\N(p_1,d)>1}}\sum_{N(p_2 p_3)<\frac{4(1+\epsilon)N(\mathcal{N})}{N(p_1)}}1
&\leq \frac{4(1+\epsilon)N(\mathcal{N})}{\phi(\mf{d})}\sum_{\substack{ N(p_1)\geq z\\p_1\mid d}}\frac{1}{N(p_1)}\\
		&\leq \frac{4(1+\epsilon)N(\mathcal{N})\omega(d)}{z\phi(\mf{d})}\ll \frac{N(\mathcal{N})^{7/8}\log N(d)}{\phi(\mf{d})}.
\end{align*}
Note that $N(p_1, d)=1$ is equivalent to $N(p_1 p_2 p_3,d)=1$ as $\mf{d}|P(y)$. Therefore, from the above bound, $r^{(k)}_{\mf{d}}$ can be reduced to 
	\begin{align}\label{rdk for B}
		r^{(k)}_{\mf{d}}=\sum_{\substack{y\leq N(p_2)\leq N(p_3)\\N(p_2 p_3)<4N(\mathcal{N})/\ell_k\\N(p_2 p_3,\mathcal{N})=1}} \ \sum_{\substack{\ell_k\leq N(p_1)<\ell_{k+1}\\p_1 p_2 p_3\equiv \mathcal{N}(\bmod d)}}1
-\frac{1}{\phi(\mf{d})}\sum_{\substack{y\leq N(p_2)\leq N(p_3)\\N(p_2 p_3)<4N(\mathcal{N})/\ell_k\\N(p_2 p_3,\mathcal{N})=1}}\sum_{\substack{\ell_k\leq N(p_1)<\ell_{k+1}\\ N(p_1 p_2 p_3,d)=1}}1
		+O\left(\frac{N(\mathcal{N})^{7/8}\log N(d)}{\phi(\mf{d})}\right).
	\end{align}
We next invoke Lemma \ref{errorbound} to bound $r^{(k)}_{\mf{d}}$ and for that we set $a(n)$ to be the characteristic function of the set 
	\begin{align*}
		\{n=p_2 p_3: y\leq N(p_2)\leq N(p_3), N(p_2 p_3, \mathcal{N})=1\}.
	\end{align*}
Letting $X=4N(\mathcal{N})/\ell_k, Y = \ell_{k+1}$,
	$Z=\ell_k$ and $a=\mathcal{N}$, we can rewrite \eqref{rdk for B} as 
	\begin{align}\label{r(k)_d_before_apply}
		r^{(k)}_{\mf{d}}=\sum_{N(n)<X}\sum_{\substack{Z\leq N(p_1)<Y\\np_1\equiv a(\bmod d)}}a(n)-\frac{1}{\phi(\mf{d})}\sum_{N(n)<X}\sum_{\substack{Z\leq N(p_1)<Y\\N(np_1,d)=1}}a(n)+O\left(\frac{N(\mathcal{N})^{7/8}\log N(d)}{\phi(\mf{d})}\right).
	\end{align}
Inserting the value of $y$, we bound $X$ in terms of $Y$ as 	 \begin{align*}
		X=\frac{4N(\mathcal{N})}{l_k}>\frac{4N(\mathcal{N})}{y}>(\log y)^{2A}\geq (\log Y)^{2A}, 
	\end{align*}
for any $A>0$. For $D:=\displaystyle\frac{N(\mathcal{N})^{1/2}}{(\log N(\mathcal{N}))^{A+1}}$ and $Q<\log N(\mathcal{N})$, in the current set up, the lower bound of $D^*$, appeared in Lemma \ref{errorbound} can be reduced to
	\begin{align*}
		D^*&=\frac{(XY)^{1/2}}{(\log Y)^A}\geq \frac{\left(\frac{4N(\mathcal{N})}{\ell_k} \ell_{k+1}\right)^{1/2}}{(\log y)^A}\geq 3^A\frac{2N(\mathcal{N})^{1/2}}{(\log N(\mathcal{N}))^A}>DQ.
	\end{align*}
Thus \eqref{r(k)_d_before_apply} and Lemma~\ref{errorbound} together bound the error term of $S(\mathcal{B}^{(k)}, P,y)$ as
	\begin{align*}
\sum_{\substack{N(\mf{d})<DQ\\\mf{d}|P(y)}}|r^{(k)}_{\mf{d}}|
&\leq \sum_{\substack{N(\mf{d})<D^*\\\mf{d}|P(y)}}|r^{(k)}_{\mf{d}}|\\
		&=\sum_{\substack{N(\mf{d})<D^*\\\mf{d}|P(y)}}\left|\sum_{N(n)<X}\sum_{\substack{Z\leq N(p_1)<Y\\np_1\equiv a(\bmod d)}}\hspace{-10pt}a(n)-\frac{1}{\phi(\mf{d})}\sum_{N(n)<X}\sum_{\substack{Z\leq N(p_1)<Y\\N(np_1,d)=1}}\hspace{-8pt}a(n)\right|+O\left(\sum_{\substack{N(\mf{d})<D^*\\\mf{d}|P(y)}}\hspace{-8pt}\frac{N(\mathcal{N})^{7/8}\log N(d)}{\phi(\mf{d})}\right)\\
		&\ll \frac{XY(\log XY)^2}{(\log Y)^A}+N(\mathcal{N})^{7/8}\log D^*\sum_{N(\mf{d})<D^*}\frac{1}{\phi(\mf{d})}\\
        &\ll \frac{XY(\log XY)^2}{(\log Y)^A}+N(\mathcal{N})^{7/8} (\log D^*)^2,		
	\end{align*}
where in the last step we have used Lemma \ref{Lem:reciprocal_sum_phi}. Now the first term of the above equation can be bounded as
	\begin{align*}
		\frac{XY(\log XY)^2}{(\log Y)^A}
		\leq \frac{4(1+\epsilon)N(\mathcal{N}) \left(\log\left(4(1+\epsilon)N(\mathcal{N})\right)\right)^2}{(\log z)^A}
		\ll \frac{N(\mathcal{N})}{(\log N(\mathcal{N}))^{A-2}},
	\end{align*}
and for the second term we use the trivial upper bound of $D^*$ as 
	\begin{align*}
		D^*<(XY)^{1/2}<N(\mathcal{N}).
	\end{align*}	
Therefore, the error term of $S(\mathcal{B}^{(k)}, P,y)$ can be written as
\begin{align*}
\sum_{\substack{N(\mf{d})<DQ\\\mf{d}|P(y)}}|r^{(k)}_{\mf{d}}|
		\ll \frac{N(\mathcal{N})}{(\log N(\mathcal{N}))^{A-2}}+N(\mathcal{N})^{7/8} (\log N(\mathcal{N}))^2
		\ll \frac{N(\mathcal{N})}{(\log N(\mathcal{N}))^{A-2}}.
\end{align*}	
We choose $A=6$ in the above equation to bound the above error term as
	\begin{align*}
		\sum_{\substack{N(\mf{d})<DQ\\\mf{d}|P(y)}}|r^{(k)}_{\mf{d}}|\ll \frac{N(\mathcal{N})}{(\log N(\mathcal{N}))^4}
	\end{align*}
for large $N(\mathcal{N})$. Finally we apply the bound \eqref{bournd of r} to bound $\tilde{R}$ as
	\begin{align*}
		\tilde{R}=\sum_{k=0}^{r}\sum_{\substack{N(\mf{d})<DQ\\\mf{d}|P(y)}}|r^{(k)}_{\mf{d}}|
		\ll \frac{\log N(\mathcal{N})}{\epsilon}\frac{N(\mathcal{N})}{(\log N(\mathcal{N}))^4}
		=\frac{\epsilon^{-1}N(\mathcal{N})}{(\log N(\mathcal{N}))^3}.
	\end{align*}
This completes the proof of the lemma. 
\end{proof}
The next lemma provides an upper bound of $S(\tilde{\mathcal{B}}, P,y)$.
\begin{lemma}\label{last_sec_first_lem}
For $y=N(\mathcal{N})^{1/3}$ and $z=N(\mathcal{N})^{1/8}$, the following identity holds: 
\begin{align*}
	S(\tilde{\mathcal{B}}, P,y)\leq \left(\frac{e^\gamma}{2}+O(\epsilon)\right)|\tilde{\mathcal{B}}|V(z)+O\left(\frac{\epsilon^{-1}N(\mathcal{N})}{(\log N(\mathcal{N}))^3}\right).
\end{align*}
\end{lemma}
\begin{proof}
For $s=\frac{\log D}{\log y}$ with $D=\frac{N(\mathcal{N})^{1/2}}{(\log N(\mathcal{N}))^7}$ and for sufficiently large $N(\mathcal{N})$, we have
\begin{align*}
	s&=\frac{\log\frac{N(\mathcal{N})^{1/2}}{(\log N(\mathcal{N}))^7}}{\log N(\mathcal{N})^{1/3}}=\frac{3}{2}+O\left(\frac{\log \log N(\mathcal{N})}{\log N(\mathcal{N})}\right)\in (1,2).
\end{align*}
The definition \eqref{F(s) and f(s)} of $F(s)$ yields 
\begin{align}\label{F(s)_for_y}
	F(s)&=\frac{2e^\gamma}{s}=\frac{4e^\gamma}{3}+O\left(\frac{\log\log N(\mathcal{N})}{\log N(\mathcal{N})}\right)=\frac{4e^\gamma}{3}+O(\epsilon).
\end{align}
It follows from Lemma~\ref{lemma_for_V(z)} that 
\begin{align}\label{V(y)}
	V(y)&=\frac{V(y)}{V(z)}V(z)=\frac{\log z}{\log y}\left(1+O\left(\frac{1}{\log N(\mathcal{N})}\right)\right)V(z)=\frac{3}{8}\left(1+O\left(\frac{1}{\log N(\mathcal{N})}\right)\right)V(z).
\end{align}
Now applying Lemma~\ref{Jurkat-Richert} with the above choice of $D$ and $y$ and then using the values of $F(s)$ and $V(y)$ from \eqref{F(s)_for_y} and \eqref{V(y)} respectively, we bound $S(\mathcal{B}^{(k)},\mathcal{P},y)$ as 
\begin{align*}
	S(\mathcal{B}^{(k)},\mathcal{P},y)&<\left(\frac{4e^\gamma}{3}+O(\epsilon)\right)|\mathcal{B}^{(k)}|\frac{3}{8}\left(1+O\left(\frac{1}{\log N(\mathcal{N})}\right)\right)V(z)+\sum_{\substack{N(\mf{d})<DQ\\\mf{d}|P(y)}}|r^{(k)}_{\mf{d}}|\nonumber\\
	&<\left(\frac{e^\gamma}{2}+O(\epsilon)\right)|\mathcal{B}^{(k)}|V(z)+\sum_{\substack{N(\mf{d})<DQ\\\mf{d}|P(y)}}|r^{(k)}_{\mf{d}}|.
\end{align*}
Finally, summing over $k$ on both sides of the above equation and applying Lemma \ref{lemma_for_R_tilde}, we conclude our lemma.
\end{proof}

In the following lemma we bound the cardinality of the set $\tilde{\mathcal{B}}$.
\begin{lemma}\label{last_sec_2nd_lemma}
For $\epsilon\in\left(0,\frac{1}{200}\right)$ small enough, we have the bound
	\begin{align*}
		|\tilde{\mathcal{B}}|<\frac{256(1+3\epsilon)cN(\mathcal{N})}{\log N(\mathcal{N})}+O\left(\frac{N(\mathcal{N})}{(\log N(\mathcal{N}))^2}\right),
	\end{align*}
	where $c$ takes the value $0.363\ldots$.
\end{lemma}
\begin{proof}
It follows from the definition of $\tilde{\mathcal{B}}$ that 
	\begin{align*}
		\tilde{\mathcal{B}}\subseteq\{\mathcal{N}-p_1 p_2 p_3:&z\leq N(p_1)<y\leq N(p_2)\leq N(p_3),N(p_1 p_2 p_3)<4(1+\epsilon)N(\mathcal{N})\}.
	\end{align*}
Thus, we can bound the cardinality of $\tilde{\mathcal{B}}$ as
	\begin{align*}
		|\tilde{\mathcal{B}}|\leq\sum_{\substack{p_1, p_2, p_3\\z\leq N(p_1)<y\leq N(p_2)\leq N(p_3)\\N(p_1 p_2 p_3)<4(1+\epsilon)N(\mathcal{N})}}1
		&\leq \sum_{\substack{p_1, p_2\\z\leq N(p_1)<y\leq N(p_2)\\N(p_1)N(p_2)^2<4(1+\epsilon)N(\mathcal{N})}}\sum_{\substack{p_3\\ N(p_3)<\frac{4(1+\epsilon)N(\mathcal{N})}{N(p_1)N(p_2)}}}1\\
		&= 4\sum_{\substack{p_1, p_2\\z\leq N(p_1)<y\leq N(p_2)\\N(p_1)N(p_2)^2<4(1+\epsilon)N(\mathcal{N})}}\pi\left(\frac{4(1+\epsilon)N(\mathcal{N})}{N(p_1)N(p_2)}\right)
\end{align*}
Applying Lemma \ref{PIT} on the above summand for sufficiently large $N(\mathcal{N})$, the above bound reduces to
\begin{align}\label{Bound of Btil}	
|\tilde{\mathcal{B}}|&<16(1+3\epsilon)\sum_{\substack{p_1, p_2\\z\leq N(p_1)<y\leq N(p_2)\\N(p_1)N(p_2)^2<4(1+\epsilon)N(\mathcal{N})}}\frac{N(\mathcal{N})}{N(p_1 p_2)\log\frac{N(\mathcal{N})}{N(p_1 p_2)}}\nonumber\\
		&=16(1+3\epsilon)N(\mathcal{N})\sum_{\substack{p_1\\z\leq N(p_1)<y}}\frac{1}{N(p_1)}\sum_{\substack{p_2\\y\leq N(p_2)<\left(\frac{4(1+\epsilon)N(\mathcal{N})}{N(p_1)}\right)^{1/2}}}\frac{1}{N(p_2)\log\frac{N(\mathcal{N})}{N(p_1 p_2)}}\nonumber\\
		&=256(1+3\epsilon)N(\mathcal{N})\sum_{\substack{\mf{p}_1\\z\leq N(\mf{p}_1)<y}}\frac{1}{N(\mf{p}_1)}\sum_{\substack{\mf{p}_2\\y\leq N(\mf{p}_2)<\left(\frac{4(1+\epsilon)N(\mathcal{N})}{N(\mf{p}_1)}\right)^{1/2}}}\frac{1}{N(\mf{p}_2)\log\frac{N(\mathcal{N})}{N(\mf{p}_1 \mf{p}_2)}},
	\end{align}
where the last step follows by taking the sums over the prime ideals generated by the Gaussian primes. We next bound the above double sum by Converting the partial sum into Riemann-Stieltjes integral twice and for that we first set $b_k$, an arithmetic function given by
\begin{align*}
b_k := \begin{cases}
			\frac{1}{N(\mf{p}_2)} & \text{ if } k=N(\mf{p}_2) \text{ for some prime ideal } \mf{p}_2 \in \mathbb{Z}[i],\\
		\ \	0 & \text{ otherwise, }
		\end{cases}
\end{align*} 
and for any $t>0$, we denote the partial sum of $b_k$ by $T(t)$ given by
	\begin{align*}
		T(t):=\sum_{\substack{\mf{p}_2\\N(\mf{p}_2)<t}}\frac{1}{N(\mf{p}_2)}=\sum_{1\leq k<t}b_k.
	\end{align*}
Letting $y_0 = \left(\frac{4(1+\epsilon)N(\mathcal{N})}{N(\mf{p}_1)}\right)^{1/2}$, the second sum in \eqref{Bound of Btil} can reduced into
\begin{align}\label{Second sum of double}
\sum_{\substack{\mf{p}_2\\y\leq N(\mf{p}_2)<y_0}}\frac{1}{N(\mf{p}_2)\log\frac{N(\mathcal{N})}{N(\mf{p}_1 \mf{p}_2)}}
		&=\int_{y}^{y_0} \frac{1}{\log\frac{N(\mathcal{N})}{N(\mf{p}_1) t}} dT(t)\nonumber\\
		&=\int_{y}^{y_0} \frac{1}{\log\frac{N(\mathcal{N})}{N(\mf{p}_1) t}} d(\log\log t) + O\left(\int_{y}^{y_0}\frac{1}{t(\log t)^2 \log\frac{N(\mathcal{N})}{N(\mf{p}_1) t}}dt\right),
	\end{align}
where the last step follows from Lemma~\ref{Merten's theorem_in_z[i]}. Applying the similar argument as in \eqref{6.2 integral}, the last integral of the above equation can be bounded as
\begin{align}\label{I2 in 7.21}
\int_{y}^{y_0}\frac{1}{t(\log t)^2 \log\frac{N(\mathcal{N})}{N(\mf{p}_1) t}}dt \ll \frac{1}{(\log N(N))^2}.
\end{align}
For the first integral, we  split the integral into two parts namely,
\begin{align*}
\int_{y}^{y_0} \frac{1}{\log\frac{N(\mathcal{N})}{N(\mf{p}_1) t}} d(\log\log t) = \int_{y}^{(N(\mathcal{N})/N(\mf{p}_1))^{1/2}} \frac{1}{\log\frac{N(\mathcal{N})}{N(\mf{p}_1) t}} d(\log\log t) + \int_{(N(\mathcal{N})/N(\mf{p}_1))^{1/2}}^{y_0} \frac{1}{\log\frac{N(\mathcal{N})}{N(\mf{p}_1) t}} d(\log\log t). 
\end{align*}
Now, making change of variable $t=\left(\frac{N(\mathcal{N})}{N(\mf{p}_1)}\right)^{1/2}r$, the second part can be bounded as
\begin{align*}
		\int_{(N(\mathcal{N})/N(\mf{p}_1))^{1/2}}^{y_0} \frac{1}{\log\frac{N(\mathcal{N})}{N(\mf{p}_1) t}} d(\log\log t)
		&=\int_{1}^{2\sqrt{1+\epsilon}}\frac{1}{r\left[\left(\log\sqrt{\frac{N(\mathcal{N})}{N(\mf{p}_1)}}\right)^2-(\log s)^2\right]}dr, \\
		&=\int_{0}^{\log(2\sqrt{1+\epsilon})}\frac{1}{\left(\log\sqrt{\frac{N(\mathcal{N})}{N(p_1)}}\right)^2-r^2}dr\\
		&\ll \frac{1}{(\log N(\mathcal{N}))^2},
\end{align*}
and for the first part, we abbreviate the integral by defining a function 
$$H(u):=\int_{y}^{(N(\mathcal{N})/u)^{1/2}}\frac{1}{\log(N(\mathcal{N})/ut)}d(\log \log t).$$
Thus, the first integral in \eqref{Second sum of double} reduces to
\begin{align}\label{I1 in 7.21}
\int_{y}^{y_0} \frac{1}{\log\frac{N(\mathcal{N})}{N(\mf{p}_1) t}} d(\log\log t) = H(N(\mf{p}_1))+O\left(\frac{1}{(\log N(\mathcal{N}))^2}\right)
\end{align}
Therefore, by inserting \eqref{I2 in 7.21} and \eqref{I1 in 7.21} into \eqref{Second sum of double}, we can rephrase the double sum in \eqref{Bound of Btil} as
	\begin{align}\label{Double sum}
		\sum_{\substack{\mf{p}_1\\z\leq N(\mf{p}_1)<y}}\frac{1}{N(\mf{p}_1)}\sum_{\substack{\mf{p}_2\\y\leq N(\mf{p}_2)<y_0}}\frac{1}{N(\mf{p}_2)\log\frac{N(\mathcal{N})}{N(\mf{p}_1 \mf{p}_2)}}
		&=\sum_{\substack{\mf{p}_1\\z\leq N(\mf{p}_1)<y}}\frac{1}{N(\mf{p}_1)}\left(H(N(\mf{p}_1))+O\left(\frac{1}{(\log N(\mathcal{N}))^2}\right)\right)\nonumber\\
		&=\sum_{\substack{\mf{p}_1\\z\leq N(\mf{p}_1)<y}}\frac{H(N(\mf{p}_1))}{N(\mf{p}_1)}+O\left(\frac{1}{(\log N(\mathcal{N}))^2}\right).
	\end{align}
For the above finite sum, we again convert into Riemann-stieltjes integral and proceed similarly as in \eqref{Second sum of double} to obtain
\begin{align*}
\sum_{\substack{\mf{p}_1\\z\leq N(\mf{p}_1)<y}}\frac{H(N(\mf{p}_1))}{N(\mf{p}_1)}
&=\int_{z}^{y}H(u)d(\log\log u)+O\left(\frac{\max{(H(z),H(y)}}{\log y}\right)\\
		&=\int_{z}^{y}H(u)d(\log\log u)+O\left(\frac{1}{(\log N(\mathcal{N})^2}\right),
\end{align*}
where in the last step we have used the fact that $H(y)=0$ and $H(z)=O\left(\frac{1}{\log N(\mathcal{N})}\right)$, which follows directly from the definition of $H(u)$. Making change of variables $t=N(\mathcal{N})^v$ and $u=N(\mathcal{N})^w$, the above integral reduces to 
	\begin{align}\label{value of c}
		\int_{z}^{y}H(u)d(\log\log u)
		=\frac{1}{\log N(\mathcal{N})}\int_{\frac{1}{8}}^{\frac{1}{3}}\hspace{-4pt}\int_{\frac{1}{3}}^{\frac{1-w}{2}}\hspace{-8pt}\frac{1}{vw(1-v-w)}dv dw
		=\frac{1}{\log N(\mathcal{N})}\int_{\frac{1}{8}}^{\frac{1}{3}}\frac{\log(2-3w)}{\log(w(1-w)}dw
		=\frac{c}{\log N(\mathcal{N})},
	\end{align}
	where 
	$c=\int_{\frac{1}{8}}^{\frac{1}{3}}\frac{\log(2-3w)}{\log(w(1-w)}dw=0.363\ldots$. Finally, after inserting \eqref{value of c} into \eqref{Double sum}, \eqref{Bound of Btil} concludes the bound of $|\tilde{\mathcal{B}}|$. This completes the proof of the lemma.
\end{proof}

\begin{lemma}\label{Estimate_S_3}
Let $y=N(\mathcal{N})^{1/3}$ and $z=N(\mathcal{N})^{1/8}$. Then for some $\epsilon>0$, we have 
\begin{align*}
	S(\mathcal{B},P,y)<\frac{4N(\mathcal{N})V(z)}{\log N(\mathcal{N})}\left(32ce^\gamma +O(\epsilon)\right)+O\left(\frac{\epsilon^{-1}N(\mathcal{N})}{(\log N(\mathcal{N}))^3}\right),
\end{align*}
where $c$ takes the value $0.363\ldots$.
\begin{proof}
The first inequality in \eqref{Inequality_B_B'} together with Lemma \ref{last_sec_first_lem} implies
\begin{align*}
	S(\mathcal{B},P,y)&<\left(\frac{e^\gamma}{2}+O(\epsilon)\right)|\tilde{\mathcal{B}}|V(z)+O\left(\frac{\epsilon^{-1}N(\mathcal{N})}{(\log N(\mathcal{N}))^3}\right).
\end{align*}
Thus invoking the bound of $|\tilde{\mathcal{B}}|$ from Lemma \ref{last_sec_2nd_lemma}, we can obtain the upper bound of $S(\mathcal{B},P,y)$. 
\end{proof}
\end{lemma}

\section{Proof of Main Theorem}\label{Proof of Main Theorem}
In this section we prove Theorem \ref{main_thm} and for that we first fix $y=N(\mathcal{N})^{1/3}$ and $z=N(\mathcal{N})^{1/8}$.
\begin{proof}[Proof of Theorem \ref{main_thm}]
The bounds in Lemma~\ref{1st sum}, Lemma~\ref{Estimate_S_2} and Lemma~\ref{Estimate_S_3} together with Proposition~\ref{lowerbound of r(N)} reduce the lower bound of $r(\mathcal{N})$ as
		\begin{align*}
			r(\mathcal{N})&\gg \frac{4N(\mathcal{N})V(z)}{\log N(\mathcal{N})}\left[\left(\frac{e^\gamma}{2}\log 3+O(\epsilon)\right)-\frac{1}{2}\left(\frac{e^\gamma \log 6}{2}+O(\epsilon)\right)-\frac{1}{128}\left(32ce^\gamma +O(\epsilon)\right)\right]\\
			&\quad+O\left(\frac{\epsilon^{-1}N(\mathcal{N})}{(\log N(\mathcal{N}))^3}\right)+O\left(N(\mathcal{N})^{\frac{7}{8}}\right)+O\left(N(\mathcal{N})^{\frac{1}{3}}\right)\\
			&=\left(2\log 3-\log 6-c+O(\epsilon)\right)\frac{e^\gamma N(\mathcal{N})V(z)}{\log N(\mathcal{N})}+O\left(\frac{\epsilon^{-1}N(\mathcal{N})}{(\log N(\mathcal{N}))^3}\right)+O\left(N(\mathcal{N})^{\frac{7}{8}}\right)+O\left(N(\mathcal{N})^{\frac{1}{3}}\right).
		\end{align*}
It follows from Lemma \ref{lemma_for_V(z)} that for $z=N(\mathcal{N})^{1/8}$, we have
		\begin{align*}
			V(z)=\frac{4}{\pi}e^{-\gamma}\mathfrak{S}_1(\mathcal{N})\frac{1}{\log N(\mathcal{N})}\left(1+O\left(\frac{1}{\log N(\mathcal{N})}\right)\right).
		\end{align*}
Thus we can write the lower bound of $r(\mathcal{N})$ as
\begin{align*}
r(\mathcal{N})&\gg \left(2\log 3-\log 6-c+O(\epsilon)\right)\frac{N(\mathcal{N})}{\left(\log N(\mathcal{N})\right)^2}\mathfrak{S}_1(\mathcal{N})\left(1+O\left(\frac{1}{\log N(\mathcal{N})}\right)\right)\\
&\quad+O\left(\frac{\epsilon^{-1}N(\mathcal{N})}{(\log N(\mathcal{N}))^3}\right)+O\left(N(\mathcal{N})^{\frac{7}{8}}\right)+O\left(N(\mathcal{N})^{\frac{1}{3}}\right).
\end{align*}			
As $2\log 3-\log 6-c>0$, we can choose $\epsilon\in\left(0,\frac{1}{200}\right)$ small enough such that 
		\begin{align*}
			(2\log 3-\log 6-c+O(\epsilon))>0.
		\end{align*}
 For this fixed $\epsilon$, we also have 
		\begin{align*}
			&\frac{\epsilon^{-1}N(\mathcal{N})}{(\log N(\mathcal{N}))^3}\ll \frac{N(\mathcal{N})}{(\log N(\mathcal{N}))^3}.
		\end{align*}
Therefore, the lower bound of $r(\mathcal{N})$ simplifies to
		\begin{align*}
			r(\mathcal{N})\gg \mathfrak{S}_1(\mathcal{N})\frac{N(\mathcal{N})}{(\log N(\mathcal{N}))^2},
		\end{align*}
which completes the proof of our main theorem.
\end{proof}	
\section{Concluding Remarks}
The main highlight of this paper was to express every even Gaussian integer with sufficiently large norm as a sum of a Gaussian prime and a Gaussian integer with at most two Gaussian prime factors, which is an analogue of Chen's representation in rational case but this result was far from the conjecture given by Holben and Jordan \cite{Holben}. 

In the same article, they even provide stronger conjectures by restricting the conditions about the Gaussian primes. The conjectures precisely state that for every even Gaussian integer $\mathcal{N}$ with norm $N(\mathcal{N})>\sqrt{2}$, there are Gaussian primes $p_1$ and $p_2$ such that $\mathcal{N} = p_1+p_2$, where the primes $p_1$ and $p_2$ make the angle $\leq \pi/4$ with $\mathcal{N}$ and for an even Gaussian integer $\mathcal{N}$ with norm $N(\mathcal{N})>\sqrt{10}$, the angle between the primes and $\mathcal{N}$ reduces to $\leq \pi/6$. It is still unknown that how much reduction of the angles between the Gaussian integer and the primes are possible for the validity of the conjecture.

\subsection*{Acknowledgements} The authors would like to thank Prof. Soumya Bhattacharya for very useful suggestions and comments. This work was initiated by the authors when the first author was an INSPIRE Faculty at IISER Kolkata. The second author is a Prime Minister Research Fellow (PMRF) at IISER Kolkata funded by Ministry of Education, Govt. of India grant 0501972. Both the authors would like to thank IISER Kolkata and IIT Kharagpur for the support during this project.


\begin{thebibliography}{99}

\bibitem{Arpit} A. Bansal, {\em Large sieve for special characters to Gaussian prime square moduli}, Funct. Approx. Comment. Math.64(2021), no.1, 23–37.

\bibitem{Bombieri} E. Bombieri, {\em On the large sieve}, Mathematika, {\bf 12} (1965), 201—225.

\bibitem{Bordelles} O. Bordellès, {\em Arithmetic tales}, Translated from the French by Véronique Bordellès Universitext Springer, London, 2012.

\bibitem{Chen} J. R. Chen, {\em On the representation of a larger even integer as the sum of a prime and the product of at most two primes}, Sci. Sinica {\bf 16} (1973), 157–176.

\bibitem{Hinz1} J. G. Hinz, {\em A generalization of Bombieri's prime number theorem to algebraic number fields}, Acta Arith, {\bf 51} (1988), no.2, 173–193.

\bibitem{Hinz2} J. G. Hinz, {\em Chen's theorem in totally real algebraic number fields}, Acta Arith, {\bf 58} (1991), no.4, 335–361.

\bibitem{Holben}
C. A. Holben, J. H. Jordan, {\em The twin prime problem and Goldbach's conjecture in the Gaussian integers}, Fibonacci Quart, {\bf 6} (1968), no.5, 81–85, 92.

\bibitem{Huxley} M.N. Huxley, {\em The large sieve inequality for algebraic number fields III. Zero-density results}, J. Lond. Math. Soc. (2), {\bf 3} (1971) 233--240.

\bibitem{Huxley1} M. N. Huxley, {\em The large sieve inequality for algebraic number fields}, Mathematika, {\bf 15} (1968), 178–187.


\bibitem{Murty} M. R. Murty, J. Esmonde, {\em Problems in Algebraic Number Theory}, second edition, Graduate Texts in Mathematics, vol.190, Springer-Verlag, New York, 2005.

\bibitem{Nathanson} M. B. Nathanson, {\em Additive Number Theory}, Springer-Verlag, New York, 1996, Graduate Texts in Mathematics, 164.

\bibitem{Pan} C. T. Pan, {\em On the representation of an even number as the sum of a prime and of an almost prime}, Acta Mathematica Sinica, {\bf 12} (1962), 95—106.

\bibitem{Rademacher}
H. Rademacher, {\em \"Uber die Anwendung der Viggo Brunschen Method auf die Theorie der algebraischen Zahlkörper}, Sitzungsber. Preuss. Akad. Wiss., Phys.-Math. Klasse (1923), 211-218.

\bibitem{Renyi} A. A. R\'enyi, {\em On the representation of an even number as the sum of a prime and an almost prime}, Izvestiya Akademii Nauk SSSR Seriya Matematicheskaya (in Russian), {\bf 12} (1948), 57--78.

\bibitem{Akshaa} A. Vatwani, {\em Bounded gaps between Gaussian primes}, J. Number Theory {\bf 171} (2017), 449--473.

\bibitem{Vinogradov} A. I. Vinogradov,
{\em The sieve method in algebraic fields. Lower bounds}, (Russian) Mat. Sb. (N.S.) {\bf 64} (106) (1964), 52–78.

\bibitem{Wang} Y. Wang, {\em On the representation of large integer as a sum of a prime and an almost prime}, Sci. Sinica, {\bf 11} (1962), 1033—1054.

\end{thebibliography}
\end{document}